\documentclass[twoside,11pt]{article}

\usepackage{blindtext}

\usepackage{jmlr2e}
\usepackage{pgfplots}
\pgfplotsset{compat=1.5}

\usepackage{lastpage}

\ShortHeadings{GAN: Dynamics}{}
\firstpageno{1}
\usepackage{amsmath,amsfonts,amssymb}
\usepackage{mathrsfs}
\usepackage{graphics,graphicx}
\usepackage{bbm}
\usepackage[margin=1in]{geometry}
\usepackage[capitalize]{cleveref}
\usepackage{xcolor}

\newcommand{\R}{\mathbb{R}}

\newcommand{\N}{\mathbb{N}}
\newcommand{\ds}{\displaystyle}

\newcommand{\Prob}{\mathcal{P}}

\newcommand{\E}{\mathbb{E}}

\begin{document}

\title{Generative Adversarial Networks: Dynamics}

\author{\name Matias G. Delgadino \email matias.delgadino@math.utexas.edu \\
       \addr Department of Mathematics\\
       University of Texas at Austin\\
       Austin, TX 78712, USA
       \AND
       \name Bruno B. Suassuna \email bruno.b.suassuna@mat.puc-rio.br \\
       \addr Departamento de Matematica\\
       Pontif\'icia Universidade Cat\'olica do Rio de Janeiro\\
       Rio De Janeiro, RJ 22451-900, Brazil
       \AND
       \name Rene Cabrera \email rene.cabrera@math.utexas.edu \\
       \addr Department of Mathematics\\
       University of Texas at Austin\\
       Austin, TX 78712, USA}


\maketitle

\begin{abstract}%
We study quantitatively the overparametrization limit of the original Wasserstein-GAN algorithm. Effectively, we show that the algorithm is a stochastic discretization of a system of continuity equations for the parameter distributions of the generator and discriminator. We show that parameter clipping to satisfy the Lipschitz condition in the algorithm induces a discontinuous vector field in the mean field dynamics, which gives rise to blow-up in finite time of the mean field dynamics. We look into a specific toy example that shows that all solutions to the mean field equations converge in the long time limit to time periodic solutions, this helps explain the failure to converge.
\end{abstract}


\vspace{0.3cm}

\section{Introduction}

\noindent Generative algorithms are at the forefront of the machine learning revolution we are currently experiencing. Some of the most famous types are diffusion models \cite{sohl2015deep}, generative language models \cite{radford2018improving} and Generative Adversarial Networks (GAN) \cite{goodfellow2014generative}. GAN was one of the first algorithms to successfully produce synthetically realistic images and audio and is the topic of this article.

A guiding assumption for GAN is that the support of the distribution can be well approximated by a lower dimensional object. That is to say although $P_*\in\Prob(\R^{K})$, we expect that the inherent correlations in data, like values of neighboring pixels in an image, drastically reduce the dimensionality of the problem. In broad terms, we expect that, in some non-specified sense, the effective dimension of the support of $P_*$ is less or equal than a latent dimension $L\ll K$. The GAN algorithm tries to find an easy way to evaluate a continuous function from $G:\R^L\to\R^K$, which we call the generator. The objective is to make $G(Z)$ to be approximately distributed like $P_*$, where $Z$ is distributed like the standard Gaussian $\mathcal{N}(0,1)\in\Prob(\R^L)$. To get an idea of orders of magnitude, \cite{karras2017progressive} creates realistic looking high resolution images  of faces with $K=1024\times 1024\times 3=3145728$ and $L=512$.

As the word adversarial in its name suggests, the algorithm pits two Neural Networks against each other, the generator network $G$ and the discriminator network $D$. The discriminator network tries to discern from the synthetic samples $G(Z)$ and the real samples $X\sim P_*$. For this purpose, the optimization over the discriminator network $D$ is the dual formulation of a metric between the associated synthetic data distribution $G\#\mathcal{N}$ and the real data distribution $P_*$. The original algorithm  \cite{goodfellow2014generative} used Jensen-Shannon divergence. The version we analyze in detail here is the Wasserstein-GAN (WGAN) \cite{arjovsky2017wasserstein} which uses the 1-Wasserstein distance instead. The behavior of GAN is known to be directly tied to the choice of the metric, see Section~\ref{sec:modecollapse} for more details.

The architecture of the Neural Networks (NN) which parametrize the generator and discriminator also plays a large role in the success of the algorithms. The paradigm for architectures at the time of the first prototypes of GANs was to use Convolutional Neural Networks (CNNs) which exploit the natural spatial correlations of pixels, see for example AlexNet introduced in \cite{krizhevsky2017imagenet}. Currently, the paradigm has changed with the advent of attention networks which are more parallelizable and outperform CNNs in most benchmarks, see \cite{vaswani2017attention}. In this paper, we forego the interesting question of the role of NN architecture to understand in more detail the induced dynamics, see Section~\ref{sec:NN} for more details. 

To understand the dynamics, we will follow the success of understanding the overparametrized limit in the supervised learning problem for shallow one hidden layer NN architectures \cite{mei2018mean,chizat2018global,rotskoff2022trainability}, see also \cite{fernandez2022continuous,9321497} for reviews of these results. In a nutshell, to the first order these articles relate Stochastic Gradient Descent (SGD) parameter training to a stochastic discretization of an associated aggregation equation \cite{bertozzi2011lp,carrillo2011global}, and to a second order to an aggregation diffusion equation \cite{carrillo2006contractions}. In probabilistic terms, this is akin to the law of large numbers \cite{sirignano2020mean} and the central limit theorem \cite{sirignano2020mean1}.

Our contribution, which is novel even in the supervised learning case, is to quantify this type of analysis. First, we show a quantitative result for the stability of the limiting aggregation equation in the 2-Wasserstein metric, see \cref{thm:1}. The difficulty of the stability in our case is not the regularity of the activation function \cite{chizat2018global}, but instead the growth of the Lipschitz constant with respect to the size of the parameters themselves. Next, we show a quantitative convergence of the empirical process to the solution to the mean field PDE, to our knowledge this is the first of its kind in terms of a strong metric like the 2-Wasserstein metric, see \cref{Approximation of parameters} and Corollary~\ref{corollary}.

Moreover, the WGAN algorithm clips the discriminator parameters after every training iteration. In the follow up work \cite{gulrajani2017improved} observed numerically that it created undesirable behavior. In terms of the mean field PDE \eqref{eq:dynamics}, the clipping of parameters induces an associated discontinuous vector field. This explains from a mathematical viewpoint the pathology mentioned before. In a nutshell, the parameter distribution can blow-up in finite time, and after that time the discriminator network loses the universal approximation capabilities, see Secion~\ref{sec:dynamics}.

Failure to converge is a known problem of GAN. For instance, \cite{karras2017progressive} introduces a progressive approach to training higher and higher resolution pictures, effectively having hot start of the algorithm at every step. By looking at an enlightening simplified example, we can explicitly understand the long time behavior of the algorithm. In this example, any initialization eventually settles to a time periodic orbit, which implies that the generator oscillates forever, see Section \ref{sec:modecollapse}.

\subsection{Outline of the paper}

The rest of the paper is organized as follows. Section~\ref{sec:mainresult} contains the notation and the main results: the well posedness of the mean field PDE system \eqref{eq:dynamics} Theorem~\ref{thm:1}, and the quantified mean field convergence Theorem~\ref{Approximation of parameters}. Section~\ref{sec:modecollapse} contains an enlightening example of the dynamics of WGAN. Section~\ref{sec:meanfield} contains the proof of Theorem~\ref{thm:1}. Section~\ref{Section Proof of Main Result} contains the proof of Theorem~\ref{Approximation of parameters}. Section~\ref{sec:conclusion} presents the conclusions and discusses some future
directions for research. Appendix~\ref{app} recalls well posedness and approximation of differential inclusions.

\section{Set up and main results}\label{sec:mainresult}
We consider a cloud of data points $\{x_i\}_{i\in I}\subset \R^{K}$, which we assume to be generated by an underlying probability measure $P_*\in\Prob(\R^{K})$. Although we do not have direct access to $P_*$, we assume the cloud of data is large enough so that we can readily sample $x_i\sim P_*$ without any inherent bias. 

The task is to generate approximate samples of the distribution $P_*$ from a base underlying probability measure which is easy to sample. We consider the Gaussian distribution $\mathcal{N}(0,1)\in\Prob(\R^L)$ in a latent space $\R^L$, where $L$ is the dimension of the latent space, which is to be chosen by the user. We will try to approximate $P_*$ by the push forward of said base distribution $G_\Theta\#\mathcal{N}$, where $G_\Theta:\R^L\to\R^K$ is a parametric function, which is usually chosen to be a Neural Network.

To choose the parameters $\Theta$, whose dimensionality we will set later, we consider the following optimization problem
$$
\inf_{\Theta} d_1(G_\Theta\#\mathcal{N},P_*),
$$
where $d_1$ is the 1-Wasserstein distance. Although this problem seems rather straight forward, the Wasserstein distance is notorious for being difficult to calculate in high dimensions, and we do not have direct access to $P_*$; hence, in practice a proxy of said distance is chosen. More specifically, we approximate the dual problem
\begin{equation}\label{eq:kant}
d_1(G_\Theta\#\mathcal{N},P_*)=\sup_{D\in \text{Lip}_1} \int_{\R^L} D( G_\Theta(z)) \;d\mathcal{N}(z)-\int_{\R^K} D(x)\;dP_*(x),   
\end{equation}
by replacing the $\text{Lip}_1$ class of functions by the parametric function $D_\Omega:\R^K\to \R$, 
$$
d_1(G_\Theta\#\mathcal{N},P_*)\sim \sup_\Omega\int_{\R^L} D_\Omega( G_\Theta(z)) \;d\mathcal{N}(z)-\int_{\R^K} D_\Omega(x)\;dP_*(x).
$$

The parametric function $D_\Omega$ will also be considered as a Neural Network and the parameters $\Omega$ are restricted to a compact convex set. The precise definition of $G_\Theta$ and $D_\Omega$ as Neural Networks with a single hidden layer is given bellow, letting $\sigma:\R\to\R$ denote the activation function. Since the parameters $\Omega$ are restricted to a compact set, if $\sigma$ is $C^1$ bounded the family $\{D_\Omega\}$ is uniformly Lipschitz.
\begin{remark}\label{rem:GAN}
    The original GAN \cite{goodfellow2014generative} utilizes the Jensen-Shannon divergence, which in terms of Legendre-Fenchel dual can be written as
    $$
        \mathbb{JS}(G_\Theta\#\mathcal{N},P_*)=\sup_{D\in C_b(\R^K)}\int_{\R^L}\log D(G_\Theta(z))\;d\mathcal{N}(z)+\int_{\R^K}\log(1-D(x))\;dP_*(x).
    $$
\end{remark}

\subsection{Neural Networks}\label{sec:NN}
For both the generator $G_\Theta$ and discriminator $D_\Omega$, we consider the simple case of a single hidden layer, which has the universal approximation property, see \cite{cybenko1989approximation}. That is to say
$$
G_\Theta(z)=\left(\frac{1}{N}\sum_{i=1}^N \alpha^1_i\sigma(\beta^1_i\cdot z+\gamma^1_i),\;...\;,\frac{1}{N}\sum_{i=1}^N \alpha^K_i\sigma(\beta^K_i\cdot z+\gamma^K_i)\right),
$$
where the array $\Theta=(\theta_1,\cdots,\theta_N)\in((\R\times\R^L\times\R)^K)^N$ is given by $\theta_i=(\alpha^j_i,\beta^j_i,\gamma^j_i)_{1\leq j\leq K}$, and $D_\Omega$ defined by
$$
D_\Omega(x)=\frac{1}{M}\sum_{i=1}^M a_i\sigma(b_i
\cdot x+c_i),
$$
where the array $\Omega=(\omega_1,\cdots,\omega_M)\in (\R\times\R^K\times\R)^M$ is given by $\omega_i=(a_i,b_i,c_i)$. To obtain rigorous quantitative estimates, throughout the paper we consider activation functions $\sigma:\R\to\R$ that are bounded in $C^2(\R)$. The typical example being the sigmoid function
$$
\sigma(u)=\frac{1}{1+e^{-u}}.
$$
Simplifying notation, we denote
$$
\alpha^j \sigma(\beta^j \cdot z+\gamma^j)=\sigma(z;\theta^j)\qquad\mbox{with}\qquad\theta^j=(\alpha^j,\beta^j,\gamma^j)\in\R\times\R^L\times\R, \qquad 1\leq j\leq K,
$$
and
$$
a \sigma(b \cdot x+c)=\sigma(x;\omega)\qquad\mbox{with}\qquad\omega=(a,b,c)\in\R\times\R^K\times\R .
$$

\begin{remark}
    The mean field analysis of two hidden layers NN is also possible, see for instance \cite{sirignano2022mean}.
\end{remark}

\subsection{Training the parameters by SGD}
We follow a simplified version of parameter training algorithm which is given in the original reference \cite{arjovsky2017wasserstein}, the only difference is that for comprehensibility we consider stochastic gradient descent instead of RMSProp (\cite{tieleman2012lecture}), see Remark~\ref{rem:RMSProp}. We use $n$ as the full step indexing, and $l$ for the sub-index related to the extra training for the Discriminator's parameters. We initialize the parameters chaotically:
$$
\Omega^{1,1}\sim \nu_{in}^{\otimes M} \qquad\mbox{and}\qquad \Theta^1\sim \mu_{in}^{\otimes N},
$$
where 
$$
\Omega^{1,1}\in (\R\times\R^K\times\R)^M\qquad\mbox{and}\qquad \Theta^1\in ((\R\times\R^L\times\R)^K)^N,
$$
and the initial distributions 
$$
\nu_{in}\in\mathcal{P}\left(\R\times\R^K\times\R\right)\qquad\mbox{and}\qquad \mu_{in}\in \mathcal{P}\left((\R\times\R^L\times\R)^K\right)
$$
are fixed independent of $N$ and $M$. Of course, correlations in parameter initialization and $N$ and $M$ dependent initial conditions can be introduced if they were desirable.

Iteratively in $n$ until convergence, and iteratively for $l=2,...,n_c$ with $n_c$ a user defined parameter, we define
$$
\Omega^{n,l}=\mathrm{clip}\left(\Omega^{n,l-1}+h\nabla_\Omega\left(D_{\Omega^{n,l-1}}(G_{\Theta^{n}}(z^n_l))-D_{\Omega^{n,l-1}}(x^n_l)\right)\right),
$$
$$
\Omega^{n+1,1}=\Omega^{n,n_c},
$$
and
$$
\Theta^{n+1}=\Theta^{n}-h\nabla_\Theta D_{\Omega^{n+1,1}}(G_{\Theta^{n}}(z^n_{n_c+1})),
$$
where the function clip stands for the projection onto $[-1,1]\times[-1,1]^K\times[-1,1]$, and $h>0$ is the learning rate which is a user chosen parameter.
The families $\{x_l^n\}_{n\in\N,l\in\{1,...,n_c\}}$, $\{z^n_l\}_{n\in\N,l\in\{1,...,n_c+1\}}$ are independent $\R^K$ and $\R^L$ valued random variables distributed by $P_*$ and $\mathcal{N}$, respectively. 
\begin{remark}
    The clipping of the parameter is made to ensure that the discriminator network is uniformly bounded in the Lipschitz norm, to approximate Kantorovich's duality formulation \eqref{eq:kant}. With this in mind, we should notice that the clipping of all parameter is slightly indiscriminate. For instance, the dependence of the discriminator function with respect to the parameter $a$ is bounded by our assumption on the activation function $\sigma$, and would not need to be clipped.
\end{remark}

\begin{remark}\label{rem:RMSProp}
    We should note that other versions of SGD like Adam or RMSProp (see \cite{kingma2014adam} and \cite{tieleman2012lecture}) are preferred by users as they are considered to outperform SGD. They introduce adaptive time stepping and momentum in an effort to avoid metastability of plateaus, and falling into shallow local minima. These tweaks of SGD add another layer of complexity which we will not analyze in this paper.
\end{remark}

\subsection{Associated Measures}
Associated to each family of parameters at the iteration step $n$ we consider the empirical measures,
\begin{eqnarray*}
    \mu_N^n&=&\frac{1}{N}\sum_{i=1}^N\delta_{\Theta^n_i}\in \mathcal{P}\left((\R\times\R^L\times\R)^K\right)\\
\nu_M^n&=&\frac{1}{M}\sum_{i=1}^M\delta_{\Omega^{n,1}_i}\in \mathcal{P}\left(\R\times\R^K\times\R\right).
\end{eqnarray*}

Abusing notation slightly and for general probability measures $\mu\in \mathcal{P}\left((\R\times\R^L\times\R)^K\right)$ and $\nu\in\mathcal{P}\left(\R\times\R^K\times\R\right)$, define 
\begin{equation}\label{eq:generator}
G_\mu(z)=\left(\int_{\R\times\R^L\times\R}\sigma(z;\theta_1)\;d\mu_1(\theta_1),\;...\;,\int_{\R\times\R^L\times\R}\sigma(z;\theta_K)\;d\mu_K(\theta_K)\right)    
\end{equation}
and
\begin{equation}\label{eq:discriminator}
D_\nu(x)=\int_{\R\times\R^K\times\R}\sigma(x;\omega)d\nu(\omega),    
\end{equation}
where $\mu_i$, for $i=1,...,K$, denotes the $i$-th marginal of $\mu$. We should note that due to the exchangeability of the parameters, there is no loss of information from considering the pair $(\Theta^n,\Omega^n)$ versus the pair $(\mu^n_N,\nu^n_M)$. In fact, using the previous notations we have
$$
G_{\Theta^n}=G_{\mu_N^n}\qquad\mbox{and}\qquad D_{\Omega^n}=D_{\nu_M^n}.
$$

Hence, to understand the behavior of the algorithm in the overparameterization limit, we will center our attention on the evolution of the empirical measures. More specifically, we consider the curves $\mu\in C\left([0,\infty);\mathcal{P}\left((\R\times\R^L\times\R)^K\right)\right)$ and $\nu\in C\left([0,\infty);\mathcal{P}\left(\R\times\R^K\times\R\right)\right)$ to be, respectively, the linear interpolation of $\mu_N^n$ and $\nu_M^n$ at the time values $t_n=n(h/N)$.

The choice of the scale $\Delta t=h/N$ is arbitrary, and could also be expressed in terms of $M$. The relationship between $N$, $M$ and $n_c$ gives rise to different mean field limits 
\begin{equation}\label{eq:speedup}
    n_c\frac{N}{M}\to\gamma_c\sim \begin{cases}+\infty\\
1\\
0,
\end{cases}
\end{equation}
and we will obtain different behavior in terms of limiting dynamics. In this paper, we address the intermediate limit $\gamma_c\sim 1$, but we should notice that in practice it is also interesting to study when $\gamma_c= \infty$, which assumes that the discriminator has been trained to convergence, see \cref{sec:modecollapse} for an illustrative example. For notational simplicity, we write the proof for $N=M$ and $n_c=1$, but our methods are valid for any finite value of $\gamma_c\sim 1$.

Explicitly, for any $t\in[0,\infty)$, we find $n\in\N$ and $s\in[0,1)$ such that 
$$
(1-s)t_n+s t_{n+1}=t
$$
and set the intermediate value as the 2-Wasserstein geodesics:
\begin{equation}\label{eq:continterp}
\mu_N(t)=\frac{1}{N}\sum_{i=1}^N\delta_{(1-s)\theta^n_i+s\theta^{n+1}_i}\qquad\mbox{and}\qquad\nu_N(t)=\frac{1}{N}\sum_{i=1}^N\delta_{(1-s)\omega^{n,1}_i+s\omega^{n+1,1}_i}.    
\end{equation}

\subsection{Identifying the limit}\label{sec:dynamics}

For a given pair of measures $\mu$ and $\nu$, consider the energy functional:
\begin{equation}
    \label{Energy Functional}
    E[\mu,\nu]= \int_{\R^L} D_\nu( G_\mu(z)) \;d\mathcal{N}(z)-\int_{\R^K} D_\nu(x)\;dP_*(x).
\end{equation}

The evolution of the limit can be characterized by the gradient descent of $E$ on $\mu$ and gradient ascent on $\nu$, the latter restricted to $\mathcal{P}([-1,1]\times[-1,1]^K\times[-1,1])$. In terms of equations we consider
\begin{equation}\label{eq:dynamics}
\begin{cases}
    \partial_t \mu-\nabla_\theta\cdot\left(\mu\; \nabla_\theta \frac{\delta E}{\delta \mu}[\mu,\nu]\right)=0,\\
    \partial_t \nu+ \gamma_c \nabla_\omega\cdot\left(\nu\; \mbox{Proj}_{\pi_{Q}}\nabla_\omega \frac{\delta E}{\delta \nu}[\mu,\nu]\right)=0,\\
    \mu(0)=\mu_{in}, \qquad \nu(0)=\nu_{in},
\end{cases}
\end{equation}
where we define $Q=[-1,1]\times[-1,1]^K\times[-1,1]$ and the first variations are
\begin{align*}
\frac{\delta E}{\delta \mu}[\mu,\nu](\theta)&=\int_{\R^L} \int_{Q}\nabla_1\sigma(G_\mu(z);\omega)\cdot \left(\sigma(z;\theta_1),\,...\,,\sigma(z;\theta_K)\right) \,d\nu(\omega)\,d\mathcal{N}(z),\\
\frac{\delta E}{\delta \nu}[\mu,\nu](\omega)&=\int_{\R^L}\sigma(G_\mu(z);\omega)\; d\mathcal{N}(z) - \int_{\R^K}\sigma(x;\omega)\; dP_*(x)
\end{align*}
and $\mbox{Proj}_{\pi_{Q}}:Q\times(\R\times\R^K\times\R)\to \R\times\R^K\times\R$ is the projection onto the tangent cone $\pi_{Q}(\omega)$. In the present case the projection can be defined by components as follows:
\begin{equation}\label{def:proj}
\mbox{Proj}_{\pi_{Q}}(\omega,V)_l=\begin{cases}
    V_l, & \omega_l\in(-1,1)\\
    V_l \frac{1-\mathrm{sign}(V_l\omega_l)}{2}, & \omega_l\in\{-1,1\}
\end{cases}   . 
\end{equation}
We should notice in fact that the projection is trivial away from the boundary, or if the vector field at the boundary points into the domain. Effectively, the projection does not allow for mass to exit the domain. We do note that this can easily make mass collapse onto the boundary and flatten the support of the distribution $\nu$ into less dimensions, see Section~\ref{sec:modecollapse} for a further discussion.

In the context of ODEs, the projection onto convex sets was considered by \cite{henry1973existence}, which we recall and expand on \cref{app}. For Hilbert spaces setting, we mention the more general sweeping processes introduced by Mureau \cite{moreau1977evolution}. Recently, projections of solutions to the continuity equation onto semi-convex subsets have been considered as models of pedestrian dynamics with density constraints, see for instance \cite{di2016measure,santambrogio2018crowd,de2016bv}.

\subsection{Main Result}
We start by showing that the mean field parameter dynamics with a discontinuous vector fields are well defined and stable. We quantify all the results with respect to the Wasserstein distance, with $d_2$ and $d_4$ representing the standard $2$-Wasserstein and $4$-Wasserstein distance, respectively.
\begin{theorem}\label{thm:1}
Given initial conditions $(\mu_{in},\nu_{in})\in \mathcal{P}\left((\R\times\R^L\times\R)^K\right)\times \mathcal{P}(Q)$ such that for some $\delta>0$
\begin{equation}\label{eq:expmoment}
\int e^{\delta |\alpha|^2}\;d\mu_{in}<\infty,    
\end{equation}
there exists a unique absolutely continuous weak solution  to the mean field system \eqref{eq:dynamics}.

Moreover, we have the following stability estimate: For any $T\in[0,\infty)$, there exists $C>1$ such that
\begin{equation}
\sup_{t\in[0,T]}d_2((\mu_1(t),\nu_1(t)),(\mu_2(t),\nu_2(t)))\le C d_4^2( (\mu_{1,in},\nu_{1,in}) , (\mu_{2,in},\nu_{2,in}) ),
\end{equation}
for any pair of weak solutions $(\mu_1,\nu_1)$ and $(\mu_2,\nu_2)$. 
\end{theorem}
The proof of \cref{thm:1} is given in \cref{sec:meanfield}, see Proposition~\ref{Stability PDE} for a precise dependence of the constants. Our main result is the following estimate on the continuous time approximation of parameter dynamics.
\begin{theorem}
\label{Approximation of parameters}
Let $(\mu_N(t),\nu_N(t))$ be the empirical measures associated to the continuous time interpolation of the parameter values, assumed to be initialized by independent samples from $(\mu_{\mathrm{in}},\nu_{in})$ given by \eqref{eq:continterp}. Consider $(\hat{\mu}_N(t),\hat{\nu}_N(t))$ the unique solution to the PDE \eqref{eq:dynamics} with random initial conditions $(\mu_N(0),\nu_N(0))$.  If $\mu_{in}$ has bounded double exponential moments on $\alpha$, that is to say for some $\delta>0$
\begin{equation}\label{eq:expmoments}
\E_{\mu_{in}}\left[e^{e^{\delta |\alpha|^{2}}}\right]<\infty, 
\end{equation}
then for any fixed time horizon $T\in[0,\infty)$ there exists $C>0$ such that
\begin{equation}\label{eq:nocurseofdimensionality}
\sup_{t\in[0,T]}\mathbb{E}d_2^2((\mu_N(t),\nu_N(t)),(\hat{\mu}_N(t),\hat{\nu}_N(t)))\leq \frac{C}{N}.    
\end{equation}
\end{theorem}
\begin{remark}
    The need for \eqref{eq:expmoments} stems from the linear dependence of the Lipschitz constant of the mean field vector field with respect to the size of the parameters, see Lemma~\ref{lem:auxV}.
\end{remark}

The proof of \cref{Approximation of parameters} is presented in \cref{Section Proof of Main Result}. Using the convergence \cref{Approximation of parameters} and the stability of the mean field \cref{thm:1}, we can obtain a convergence rate estimate which suffers the curse of dimensionality.
\begin{corollary}\label{corollary}
Under the hypotheses of \cref{thm:1} and \cref{Approximation of parameters}, for any fixed $T>0$, there exists $C>0$ such that
\begin{equation}\label{eq:curseofdim}
    \max_{t\in [0,T]}\E d^2_2((\mu(t),\nu(t)),(\mu_N(t),\nu_N(t)))\le \frac{C}{N^{\frac{2}{K(L+2)}}} 
\end{equation}
where $(\mu,\nu)$ is the unique solution of \eqref{eq:dynamics} and $(\mu_N,\nu_N)$ is the curve of interpolated empirical measures associated to the parameter training \eqref{eq:continterp}.
\end{corollary}

\begin{remark}
We should note that the difference between the results of \cref{Approximation of parameters} and Corollary~\ref{corollary} is that the estimate \eqref{eq:nocurseofdimensionality} does not suffer from the curse of dimensionality, while the stronger estimate \eqref{eq:curseofdim} does. The later dependence on dimension is typical and sharp for the approximation of the Wasserstein distance with sampled empirical measures, see \cite{dudley1978central,fournier:hal-00915365,bolley2007quantitative}. This stiff dependence on dimension suggests that studying the long time behavior of the mean field dynamics of smooth initial data $(\mu(t),\nu(t))$ is not necessarily applicable in practice. Instead, the focus should be to show that with high probability that discrete mean field trajectories $(\hat{\mu}_N(t),\hat{\nu}_N(t))$ converge to a desirable saddle point of the dynamics. See \cref{sec:modecollapse} for an explicit example of long time behavior.
\end{remark}

\begin{proof} [Proof of Corollary~\ref{corollary}]
We consider the auxiliary pair of random measure-valued paths $(\hat{\mu}_N,\hat{\nu}_N)$ which are a solution to \eqref{eq:dynamics} with stochastic initial conditions $(\mu_N(0),\nu_N(0))$, that is
$$
\hat{\mu}_N(0)=\mu_N(0)=\frac{1}{N}\sum_{i=1}^N\delta_{\theta_{i,in}}\qquad\mbox{and}\qquad\hat{\nu}_N(0)=\nu_N(0)=\frac{1}{N}\sum_{i=1}^N\delta_{\omega_{i,in}},
$$
where $\theta_{i,in}$ and $\omega_{i,in}$ are $N$ independent samples from $\mu_{in}$ and $\nu_{in}$, respectively.

By the large deviation estimate in \cite{fournier:hal-00915365}, for $q$ large enough we have
$$
\E[d_4^4((\hat{\mu}_N(0),\hat{\nu}_N(0)),(\mu_{in},\nu_{in}))] \leq C M_q^{\frac{4}{q}}\left(\frac{1}{N^{\frac{4}{K(L+2)}}} + \frac{1}{N^{\frac{q-4}{q}}}\right),
$$
where $M_q$ denotes the $q$-th moment of $\mu_{in}\otimes\nu_{in}$. By \cref{thm:1}, taking $q$ large enough, and using that $\mu_{in}\otimes \nu_{in}$ has finite moments of all orders we have
$$
\E[d_2^2((\hat{\mu}_N(t),\hat{\nu}_N(t)),(\mu(t),\nu(t))] \leq C\E[d_4^2((\hat{\mu}_N(0),\hat{\nu}_N(0)),(\mu(0),\nu(0)))]\le \frac{C}{N^{\frac{2}{K(L+2)}}}.
$$
By the triangle inequality,
$$
 d_2((\mu(t),\nu(t)),(\mu_N(t),\nu_N(t))) \leq d_2((\mu_N(t),\nu_N(t)),(\hat{\mu}_N(t),\hat{\nu}_N(t)) +
d_2((\hat{\mu}_N(t),\hat{\nu}_N(t)),(\mu(t),\nu(t)),
$$
so the result follows by the previous estimate and  \cref{Approximation of parameters}.
\end{proof}

\section{Mode Collapse and Oscillatory Behavior}\label{sec:modecollapse}
A standard problem of GANs is known as mode collapse,  \cite{srivastava2017veegan,metz2016unrolled,thanh2020catastrophic}. This can be broadly described as the generator outputting only a small subset of the types of clusters that are present in the original distribution. Although the generator outputs a convincing sample if considered individually, the overall distribution of samples is off. An extreme example is when the generator outputs almost identical samples for any value of the latent variable $z$.

An explanation of this behavior for the original GAN algorithm is the use of Jensen-Shannon divergence $\mathbb{JS}(G\#\mathcal{N},P_*)$, see Remark~\ref{rem:GAN}. More specifically, if the measures are mutually singular $G\#\mathcal{N}\perp P_*$, then $\mathbb{JS}(G\#\mathcal{N},P_*)=\log(2)$ independently of how close the supports are to each other. Namely, the gradient of the associated loss function vanishes and there are no local incentives for the generator to keep learning. As we are not expecting for the support of these measures to be absolutely continuous, in fact we are postulating that in some sense the dimension of the support of $G\#\mathcal{N}$ is smaller than $L\ll K$, this case is more likely to be normal than the exception.

The W-GAN \cite{arjovsky2017wasserstein} and its improved variant \cite{gulrajani2017improved} try to fix this by considering 1-Wasserstein distance instead which does not suffer from the vanishing gradient problem. Still, training the Generator to get useful outputs is not an easy task, it requires a lot of computation time and more often than not it fails to converge. For instance to produce realistic looking images \cite{karras2017progressive} took 32 days of GPU compute time, and the networks are trained on progressively higher and higher resolution images to help with convergence.

\subsection{An explicit example}

Consider a bimodal distribution as the toy example: 
$$
P_*=\frac{1}{2}\delta_{-1}+\frac{1}{2}\delta_1\in\Prob(\R).
$$ 
We consider the simplest network that can approximate this measures perfectly. We consider the generator, depending on a single parameter $g\in\R$ to be given by
$$
G(z,g)=\begin{cases}
    -1 & x<g\\
    1 & x>g.
\end{cases}
$$
Although, this generator architecture seems far from our assumptions \ref{sec:NN}. This type of discontinuity arises naturally as a limit when the parameters go to infinity. Namely, if we take $b,\,c\to\infty$ in such a way that $c/b\to g\in \R$, then
$$
\sigma(b x+c)\to \begin{cases}
    0 & x<g\\
    1 & x>g,
\end{cases}
$$
where $\sigma$ is the sigmoid. The generator $G$ can then be recovered as a linear combination of two such limits. The generated distribution is given by
$$
G_g\#P=\Phi(g)\delta_{-1}+(1-\Phi(g))\delta_{1},
$$
where $\Phi(g)=P(\{z<g\})$ is the cumulative distribution function of the prior distribution $P\in\mathcal{P}(\R)$, which we can chose. We make the choice of the cumulative distribution
$$
\Phi(g)=\frac{1}{1+e^{-g}}\qquad\mbox{for $g\in\R$}
$$
to simplify the calculations. Under this choice for $g=0$, we have that $G_g\#P=P_*$, hence the network can approximate the target measure perfectly.

Moreover, we can explicitly compute the 1-Wasserstein distance,
$$
d_1(G_g\#P,P_*)= \left|\frac{1}{2}-\Phi(g)\right|,
$$
see the figure below.
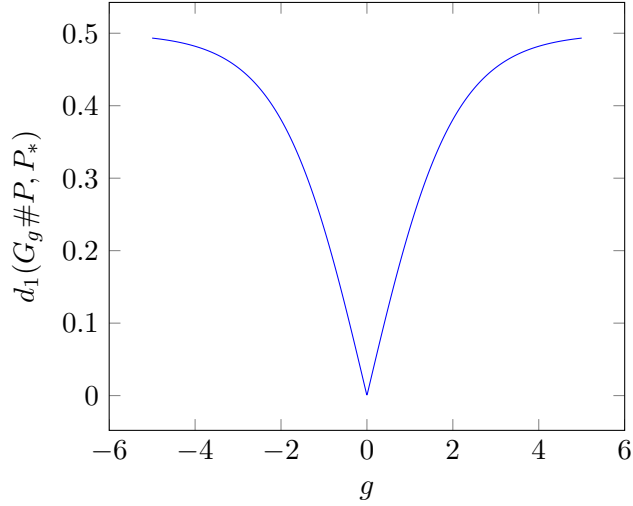
\begin{center}
\begin{tikzpicture}\label{fig1}
\begin{axis}[
xlabel={$g$},
ylabel={$d_1(G_g\#P,P_*)$},
]
\addplot [blue] table {cdf.dat};
\end{axis}
\end{tikzpicture}
\end{center}
We can clearly see that this function has a unique minimum at $g_*=0$, and also that this function is concave in $g$ away from $g=g_*$. The concavity of the functional makes the problem more challenging from the theoretical perspective and it will explain the oscillatory behavior of the algorithm close to the minimizer $g_*$.

For the discriminator, we consider a ReLU activation given by
$$
D(x;\omega)=(\omega x)_+
$$
with $\omega\in[-1,1]$. We note that taking a single parameter, instead of a distribution, for the discriminator is supported by the mean field dynamics \eqref{eq:dynamics}. In the sense that under a bad initialization of parameters, the parameters of the discriminator can blow up in finite time to $\nu=\delta_{\omega(t)}$. 

We consider the joint dependence function
\begin{eqnarray*}
    \Psi(\omega,g)&=&\int_{\R} D_\omega(G_g(z))\;dP(z)-\int_{\R} D_\omega(x)\;dP_*(x)\\
    &=& \Phi(g)(-\omega)_++(1-\Phi(g))(\omega)_+-\frac{1}{2}(-\omega)_+-\frac{1}{2}(\omega)_+\\
     &=& \left(\frac{1}{2}-\Phi(g)\right)\omega.
\end{eqnarray*}


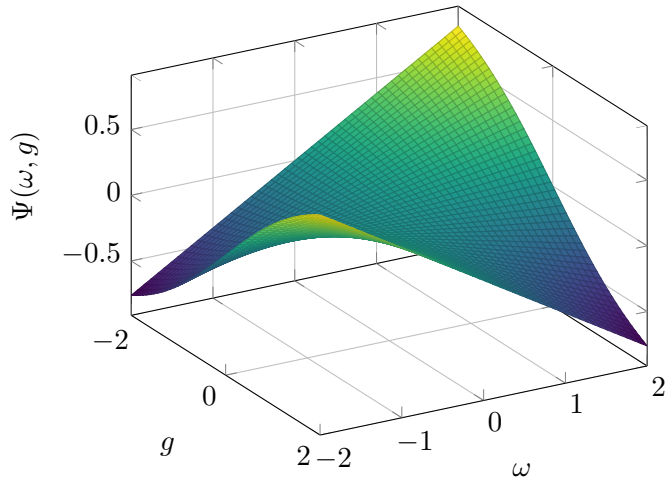
\begin{center}
\begin{tikzpicture}
    \begin{axis}[
        view={60}{30}, 
        xlabel={$g$},
        ylabel={$\omega$},
        zlabel={$\Psi(\omega,g)$},
        domain=-2:2,
        y domain=-2:2,
        samples=50,
        samples y=50,
        grid=major,
    ]
    \addplot3[
        surf,
        mesh/rows=50,
        colormap/viridis
    ] { (0.5 - 1 / (1 + exp(-x))) * y };
    \end{axis}
\end{tikzpicture}
\end{center}

Ignoring, for now, the projection onto $\omega\in[-1,1]$, we have the dynamics
$$
\left\{\begin{array}{rcccl}
\ds\dot{g}(t)&=&\ds -\nabla_{g}\Psi[g, \omega]&=&\ds\frac{1}{2}\frac{e^{-g}}{(1+e^{-g})^{2}}\omega \\ 
\ds\dot{\omega}(t)&=&\ds\gamma_c \nabla_{\omega}\Psi[g, \omega] &=&\ds\frac{\gamma_c}{2}\frac{e^{-g}-1}{1+e^{-g}},
\end{array}\right.
$$
where $\gamma_c$ is the critics speed up \eqref{eq:speedup}. These dynamics can be integrated perfectly, to obtain that
$$
E_{\gamma_c}(\omega(t),g(t))=2\cosh(g(t))+\frac{|\omega(t)|^2}{\gamma_c}=2\cosh(g_{in})+\frac{|\omega_{in}|^2}{\gamma_c}=E_{\gamma_c}(\omega_{in},g_{in}).
$$

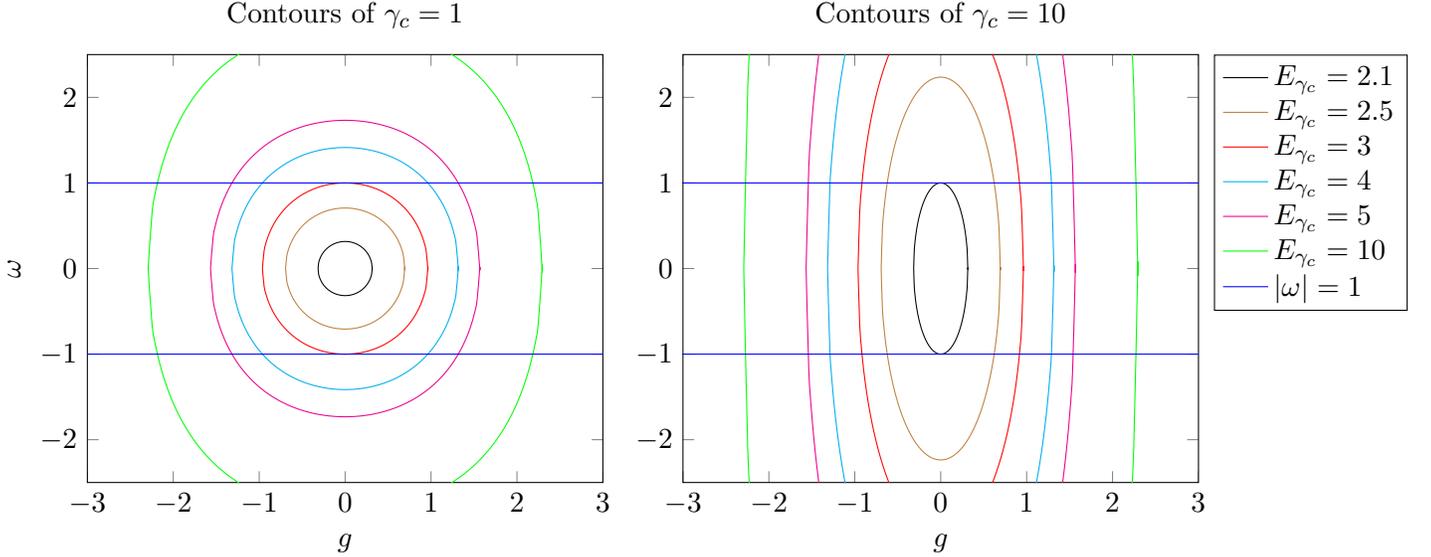
\hspace{-1.7cm}\begin{tikzpicture}\label{fig1}
\begin{axis}[
        title={Contours of \( \gamma_c=1 \)},
        xlabel={$g$},
        ylabel={$\omega$},
        xmin=-3, xmax=3,
        ymin=-2.5, ymax=2.5,
    ]
\addplot [black, smooth, mark=none] table {data2.1.table};
\addplot [brown, smooth, mark=none] table {data2.5.table};
\addplot [red, smooth, mark=none] table {data3.table};
\addplot [cyan, smooth, mark=none] table {data4.table};
\addplot [magenta, smooth, mark=none] table {data5.table};
\addplot [green, smooth, mark=none] table {data10.table};
\addplot [blue] {1};
\addplot [blue] {-1};

\end{axis}
\end{tikzpicture}
\begin{tikzpicture}\label{fig1}
\begin{axis}[
        title={Contours of \( \gamma_c=10 \)},
        xlabel={$g$},
        legend cell align=left,
        legend pos=outer north east,
        xmin=-3, xmax=3,
        ymin=-2.5, ymax=2.5,
    ]
\addplot [black, smooth, mark=none] table {10data2.1.table};
\addplot [brown, smooth, mark=none] table {10data2.5.table};
\addplot [red, smooth, mark=none] table {10data3.table};
\addplot [cyan, smooth, mark=none] table {10data4.table};
\addplot [magenta, smooth, mark=none] table {10data5.table};
\addplot [green, smooth, mark=none] table {10data10.table};
\addplot [blue] {1};
\addplot [blue] {-1};

\legend{$E_{\gamma_c}=2.1$,$E_{\gamma_c}=2.5$, $E_{\gamma_c}=3$,$E_{\gamma_c}=4$,$E_{\gamma_c}=5$, $E_{\gamma_c}=10$,$|\omega|=1$}
\end{axis}
\end{tikzpicture}

In the figure above, we plot the level sets of $E_{\gamma_c}$ as well as the restriction of $|\omega|\le 1$. We notice that, given the value of $\gamma_c$ there exists a unique level set 
$$
E_*(\gamma_c)=2+\frac{1}{\gamma_c}
$$ 
such that the level set $\{E_{\gamma_c}=E_*\}$ is tangent to the restriction $|\omega|\le 1$. 

Now, we consider the dynamics with the restriction $|\omega|\le 1$. We notice that for any initial conditions $(\omega_{in},g_{in})$ satisfying $E_{\gamma_c}(\omega_{in},g_{in})\le E_*(\gamma_c)$ the trajectory of parameters is unaffected by the restriction $|\omega|\le 1$ and it is time periodic. On the other hand, if we consider initial conditions $(\omega_{in},g_{in})$ satisfying $E_{\gamma_c}(\omega_{in},g_{in})> E_*(\gamma_c)$ and $|\omega_{in}|\le 1$, the trajectory will follow the unconstrained dynamics until it hits the boundary of the restriction $\omega(t)\in \partial Q=\{|\omega|= 1\}$. Then it follows on the boundary $\omega(t)\in\partial Q=\{|\omega|=1\}$ until it reaches the point $(\omega(t_*),g(t_*))=(\pm 1,0)$ on the tangential level set $\{E_{\gamma_c}=E_*\}$ and start following this trajectory becoming time periodic. Hence, there exists $t_*=t(E_{\gamma_c}(\omega_{in},g_{in}))$ large enough, such that the trajectory $(\omega(t),g(t))\in \{E_{\gamma_c}(\omega_{in},g_{in})=E_*\}$ for $t>t_*$.
Therefore, we can conclude that
$$
|g(t)|\le \cosh^{-1}\left(1+\frac{1}{2\gamma_c}\right)\qquad\forall t>t_*.
$$
Looking back at the figure, we can see that for $\gamma_c=1$ that the limiting trajectory is $\{E_{1}=3\}$, and that the generator parameter oscillates in the range $|g(t)|\le 0.96$ for $t>t_*$. While for $\gamma_c=10$, we obtain that the limiting trajectory is $\{E_{10}=2.1\}$ and the limiting oscillations are smaller
$|g(t)|\le 0.31$ for $t>t_*$.

We do notice that regardless of the parameter $\gamma_c$ and the initial configuration, the limiting trajectory is always periodic in time. In fact, we expect that every trajectory of the mean field dynamics settles into a periodic solution.

\section{Properties of the mean field}\label{sec:meanfield}

One of the main theoretical obstructions to understand the well-posedness of this flow is that the projection operator $\mbox{Proj}_{\pi_{Q}}$ induces a discontinuous vector field in \eqref{eq:dynamics}. Nevertheless, the convexity of the domain $Q=[-1,1]\times[-1,1]^K\times[-1,1]$ can be leveraged to obtain a stability estimate.

Given a time dependent continuous vector field $V:[0,\infty)\times Q\to \R\times\R^K\times\R$, its projection $\mbox{Proj}_{\pi_{Q}}V_t$ is a Borel measurable vector field which is square integrable in space and time for any finite time horizon $T>0$ and curve of probability measures $\nu\in C([0,\infty),\mathcal{P}(Q))$,
$$
    \int_0^T\left(\int_Q |\mbox{Proj}_{\pi_{Q}}V_t|^2\;d\nu_t\right)\;dt<\infty.
$$
Hence, as long as the underlying velocity field inducing the motion is continuous, we can consider the notion of weak solution for the continuity equation given by \cite[Chapter 8]{ambrosio2005gradient}. 


With this in mind, we first notice the Lipschitz continuity properties of the vector fields that induce the motion \eqref{eq:dynamics}. More specifically, we denote by
\begin{equation}\label{eq:Vtheta}
    V^\Theta_{(\mu,\nu)}(\theta)= - \nabla_\theta \frac{\delta E}{\delta \mu}[\mu,\nu](\theta) =\mathbb{E}_zv^\Theta_{(\mu,\nu)}(\theta,z)
\end{equation}
and
\begin{equation}\label{eq:Vomega}
V^\Omega_{(\mu,\nu)}(\omega)= \nabla_\omega \frac{\delta E}{\delta \nu}[\mu,\nu](\omega) = \mathbb{E}_z\mathbb{E}_xv^\Omega_{(\mu,\nu)}(\omega,z,x),
\end{equation}
where we define the vector fields
\begin{equation}\label{eq:vomega}
v^\Theta_{(\mu,\nu)}(\theta,z) = - \nabla_{\theta}\int_{[-1,1]^{1+K+1}}\nabla_1\sigma(G_\mu(z);\omega)\cdot \left(\sigma(z;\theta_1),\,...\,,\sigma(z;\theta_K)\right)d\nu(\omega)
\end{equation}
and
\begin{equation}\label{eq:vtheta}
v^\Omega_{(\mu,\nu)}(\omega,z,x) = \nabla_{\omega} [\sigma(G_\mu(z);\omega)\; - \sigma(x;\omega)].
\end{equation}
In Lemma~\ref{lem:auxV} below, we show that $V_{(\mu,\nu)}^\Theta(\theta)$ and $V_{\mu,\nu}^\Omega(\omega)$ are Lipschitz continuous with respect to the dependence of arguments $\theta,$ $\omega$ as well as the measure arguments $(\mu,\nu)$. Notice that $V^\Omega$ and $v^\Omega$ do not depend on $\nu$, only on $\mu$.

By \cite[Theorem 8.2.1]{ambrosio2005gradient}, any continuous solution to the continuity equation \eqref{eq:dynamics} is supported over solutions of the associated characteristic field. Using the classical theory \cite{henry1973existence} for projected ODE flows, we can show that the characteristic equations
\begin{equation}\label{eq:projectedODEs}
\begin{cases}
    \frac{d}{dt}(\theta,\omega)=(V^\Theta_{(\mu,\nu)}(\theta_i),\mbox{Proj}_{\pi_Q(\omega_i)}V^\Omega_{(\mu,\nu)}(\omega_i))\\
    (\theta,\omega)(0)=(\theta_{in},\omega_{in})
\end{cases}    
\end{equation}
have a unique solution. More specifically, an absolutely continuous curve $(\mu,\nu)\in AC([0,\infty);\mathcal{P}\left((\R^{L+2})^{K}\right)\times \mathcal{P}(Q))$ is a weak solution to \eqref{eq:dynamics}, if it is given as the image of the initial distributions $(\mu_{in},\nu_{in})$ through the unique projected ODE flow. That is to say, 
\begin{equation}\label{eq:mildsolution}
(\mu,\nu)(t)=\Phi^t_{(\mu,\nu)}\#(\mu_{in},\nu_{in})    
\end{equation}
the family of continuous mappings $\Phi^t_{\mu,\nu}:(\R\times \R^{L}\times \R)^{K}\times Q \to (\R\times \R^{L}\times \R)^{K}\times Q$ given by 
\begin{equation}\label{eq:diffeos}
    \Phi^t_{(\mu,\nu)}(\omega_{in},\theta_{in})=(\omega(t),\theta(t)),   
\end{equation}
where $(\omega(t),\theta(t))$ is the unique Lipschitz solution to \eqref{eq:projectedODEs}.

One of the main technical hurdles is that the vector fields inducing the motion are only locally Lipschitz. The Lipschitz constant depends itself on the size of $\alpha$, which is the first variable of $\theta$. Hence, to obtain a stability estimates we need to measure the distance of the initial condition in a $p$-Wasserstein distance with $p>2$. The choice of $p=4$ in the following result is arbitrary.
\begin{proposition}[Stability]\label{Stability PDE} Assume $(\mu_1,\nu_1),(\mu_2,\nu_2)\in AC\left([0,\infty);\mathcal{P}\left((\R^{L+2})^{K}\right)\times \mathcal{P}(Q)\right)$ are weak solutions to \eqref{eq:dynamics} which satisfies \eqref{eq:mildsolution}. Assume that the initial distribution has bounded exponential moments in the following sense: there exists $\delta>0$ such that
$$
\int e^{\delta |\alpha|^2}\;d\mu_{1,in},\,\int e^{\delta |\alpha|^2}\;d\mu_{2,in}<\infty.
$$
Then for any $t>0$ we have the bound 
\begin{equation}\label{eq:stabilitymetric}
d_2((\mu_1(t),\nu_1(t)),(\mu_2(t),\nu_2(t)))\le A(t)e^{B(t)}d_4^2( (\mu_{1,in},\nu_{1,in}) , (\mu_{2,in},\nu_{2,in}) ),
\end{equation}
where
$$
A(t) = e^{C\left(t^2 + t\Lambda \right)}\left( \int e^{Ct|\alpha|}d\mu_{\mathrm{1,in}}+ \int e^{Ct|\alpha|}d\mu_{\mathrm{2,in}}\right)^{1/2},
$$
and 
$$
B(t) = C tA(t)\left(  t + \Lambda \right),
$$
with
$$
\Lambda = 1 + \left(\int |\alpha|^2 d\mu_{1,\mathrm{in}}\right)^{1/2} + \left(\int |\alpha|^2 d\mu_{2,\mathrm{in}}\right)^{1/2}
$$
and $C>0$ a constant that only depends on $\|\sigma\|_{C^2}$.
\end{proposition}
\begin{remark}
    The double exponential growth on the estimate is related to the  dependence Lipschitz constant of the vector field with respect to the size of the parameters themselves, see Lemma~\ref{lem:auxV} for the specific estimates.
\end{remark}    

For discrete initial conditions, existence to \eqref{eq:dynamics} follows from applying the results in \cref{app}. Using stability, we can then approximate the initial condition by taking discrete approximations of it.
\begin{proposition}[Existence]\label{prop:existence}
For any initial condition $(\mu_{in},\nu_{in})\in\mathcal{P}\left((\R^{L+2})^{K}\right)\times \mathcal{P}(Q)$ satisfying that there exists $\delta>0$ such that
$$
\int e^{\delta |\alpha|^2}\;d\mu_{in}<\infty,
$$
there exists $(\mu,\nu)\in AC\left([0,\infty);\mathcal{P}\left((\R^{L+2})^{K}\right)\times \mathcal{P}(Q)\right)$ a weak solution to \eqref{eq:dynamics} which satisfies the mild formulation \eqref{eq:mildsolution}.  
\end{proposition}
\begin{proof}[Proof of Proposition~\ref{prop:existence}]
    For any $L\in\N$ we consider a deterministic discretization 
    $$
    \mu^L_{in}=\sum_{i=1}^Lw_i\delta_{\theta^i_L}
    \,\qquad\mbox{and}\qquad\nu^L_{in}=\sum_{i=1}^Lv_i\delta_{\omega^i_L}
    $$ 
    of the initial conditions 
    $\mu_{in},\,\nu_{in}$, where $w_i$ and $v_i$ are weights which add up to 1. The main properties we need from this discretization is that
    $$
        \lim_{L\to\infty}d_4((\mu^L_{in},\nu^L_{in}),(\mu_{in},\nu_{in}))=0 \qquad\mbox{and}\qquad \int e^{\delta |\alpha|^2}\;d\mu^L_{in}\le \int e^{\delta |\alpha|^2}\;d\mu_{in}.
    $$
    Such a discretization can be given by the following procedure. For simplicity we consider $R=2^{k(L+2)K}$, we divide the box $[-\log R,\log R]^{(L+2)K}$ into equal sized boxes $\{B_i\}_{i=1}^L$. We assign $\theta^i$ to be the the point with the smallest norm of the box $B_i$, and the weights are given by $w_i=\mu_{in}(B_i)$. We add any leftover mass on $([-\log R,\log R]^{(L+2)K})^c$ to the delta at the origin. We do the same to produce $\nu_{in}^L$.

    By \cref{app}, for any $L\in\N$ there exists a unique solution to the projected ODE associated to the solution of the mean field equations 
    with initial conditions given by $(\mu^L_{in},\nu^L_{in})$. Hence, we can construct a global weak solution to the PDE $(\mu^L(t),\nu^L(t))$. By the stability result, we know that for any finite time horizon $T>0$, $\{(\mu^L,\nu^L)\}_{L}$ form a Cauchy sequence in $AC([0,T],\mathcal{P}\left((\R^{L+2})^{K}\right)\times \mathcal{P}(Q))$. Hence, there exists $(\mu,\nu)\in AC\left([0,\infty);\mathcal{P}\left((\R^{L+2})^{K}\right)\times \mathcal{P}(Q)\right)$, such that for any fixed time horizon $T$
    $$
    \lim_{L\to\infty} \sup_{t\in[0,T]}d_2^2((\mu(t),\nu(t)),(\mu^L(t),\nu^L(t))=0.
    $$
    By Lemma~\ref{lem:growth}, $\mu^L$ satisfies the growth condition \eqref{eq:growth}, and so does $\mu$. By Lemma~\ref{lem:stability1}, we have that the associated projected ODE flows also converge
    \begin{eqnarray}
        \ds \lim_{L\to\infty}\sup_{t\in[0,T]}|\Phi^t_{(\mu^L,\nu^L)}(\theta,\omega)-\Phi^t_{(\mu,\nu)}(\tilde\theta,\tilde\omega)|^2 \leq e^{C(\Lambda+|\alpha|+|\tilde\alpha|)} |(\theta,\omega)-(\tilde\theta,\tilde{\omega})|^2.
    \end{eqnarray}
    Using that 
    $$
        (\mu^L(t),\nu^L(t))=\Phi^t_{(\mu^L,\nu^L)}\#(\mu_{in}^L\nu_{in}^L)
    $$
    and the uniform exponential integrability of $(\mu_{in}^L,\nu_{in}^L)$, we can conclude that
    $$
        (\mu(t),\nu(t))=\Phi^t_{(\mu,\nu)}\#(\mu_{in},\nu_{in}),
    $$
    which in turn implies that $(\mu,\nu)$ is a weak solution to \eqref{eq:dynamics} satisfying \eqref{eq:mildsolution}.
    \end{proof} 

For the next lemma we use the notation $\theta=(\theta_1,\ldots,\theta_K)$ with $\theta_i=(\alpha_i,\beta_i,\gamma_i)\in \R\times\R^L\times \R$, and $\alpha=(\alpha_1,\ldots,\alpha_K)\in \R^K$.
\begin{lemma}\label{lem:auxV} There exists $C\in\R$ depending on $\|\sigma\|_{C^1}$ such that the vector fields \eqref{eq:Vtheta}, \eqref{eq:Vomega}, \eqref{eq:vomega} and \eqref{eq:vtheta} satisfy the bounds
\begin{equation}\label{eq:infty1}
\left\|V^\Omega_{(\mu,\nu)}\right\|_{\infty}\leq C\left(1+\int |\alpha|d\mu\right),\qquad \left\|v^\Omega_{(\mu,\nu)}(\cdot,z,x)\right\|_\infty\leq C\left(1+|x|+\int |\alpha|d\mu\right),
\end{equation}
and 
\begin{eqnarray}
    \left\|\left(V^\Theta_{j}\right)_r\right\|_{\infty}&\leq& \begin{cases}
    C &\mbox{for $r=1$}\\
    C |\alpha_j| &\mbox{for $r\ne 1$},
\end{cases}\nonumber\\
\left\|\left(v^\Theta_{j}\right)_r\right\|_{\infty}&\leq& \begin{cases}
C &\mbox{for $r=1$}\\
C|\alpha_j| \left(1+|z|\right)&\mbox{for $r\ne 1$},
\end{cases}\label{eq:infty2}
\end{eqnarray}
where $(v_j)_r$ denotes the $r$-th component of the $j$-th position. 

Moreover, we have the following Lipschitz estimate. There exists $C\in\R$ depending on $\|\sigma\|_{C^2}$, such that
$$
\begin{array}{l}
     \ds\left|V^\Theta_{(\mu_1,\nu_1)}(\theta)-V^\Theta_{(\mu_2,\nu_2)}(\tilde{\theta})\right|\le C(|\alpha|+|\tilde\alpha|+A(\mu_1,\mu_2))\left(d_2((\mu_1,\nu_1),(\mu_2,\nu_2))+|\theta-\tilde{\theta}|\right)
\end{array}
$$
$$
|V^\Omega_{(\mu_1,\nu_1)}(\omega_1)-V^\Omega_{(\mu_2,\nu_2)}(\omega_2)|\le CA(\mu_1,\mu_2)\big(|\omega_1-\omega_2|+d_2(\mu_1,\mu_2)\big),
$$
and
$$
\begin{array}{l}
\ds |v^\Theta_{(\mu_1,\nu_1)}(\theta,z)-v^\Theta_{(\mu_2,\nu_2)}(\tilde{\theta},z)|\\
\ds \qquad\le C\left(|\alpha|+|\tilde{\alpha}|+A(\mu_1,\mu_2)+ |z|\right)\big(d_2((\mu_1,\nu_1),(\mu_2,\nu_2))+|\theta-\tilde{\theta}|\big),
\end{array}
$$
$$
|v^\Omega_{(\mu_1,\nu_1)}(\omega_1,z,x)-v^\Omega_{(\mu_2,\nu_2)}(\omega_2,z,x)|\le C\left(A(\mu_1,\mu_2) + |x|+|z|\right)\big(|\omega_1-\omega_2|+d_2(\mu_1,\mu_2)\big),
$$
where 
$$
A(\mu_1,\mu_2) = 1+\left(\int |\alpha|^2d\mu_1\right)^{1/2} + \left( \int |\alpha|^2d\mu_2\right)^{1/2},
$$
and
$\alpha_1$, $\alpha_2$ are the first components of $\theta_1$, $\theta_2$, respectively.
\end{lemma}

\begin{proof}[Proof of Lemma~\ref{lem:auxV}] Throughout the proof, we use the notation $\theta=(\theta_1,...,\theta_K)$ with $\theta_i=(\alpha_i,\beta_i,\gamma_i)\in \R\times\R^L\times \R$, $\alpha=(\alpha_1,...,\alpha_K)\in (\R^L)^K$, and $\omega=(a,b,c)\in Q$. We begin by explicitly writing out the vector fields
$$
 v^\Omega_{(\mu,\nu)}(\omega,z,x) = \begin{pmatrix}
    \sigma(b\cdot G_\mu(z)+c) - \sigma(b\cdot x + c)\\
    a G_\mu(z)\sigma'(b\cdot G_\mu(z)+c) - a x \sigma'(b\cdot x+c)\\
    a\sigma'(b\cdot G_\mu(z)+c)-a\sigma'(b\cdot x + c)
\end{pmatrix},   
$$
and $v^\Theta_{(\mu,\nu)}(\theta,z) = \left( v^\Theta_{(\mu,\nu);1}(\theta_1,z), \cdots, v^\Theta_{(\mu,\nu);K}(\theta_K,z)\right)$ with for $1\leq j\leq K$:
$$
    v^\Theta_{(\mu,\nu);j}(\theta_j,z) = - \int_{[-1,1]^{1+K+1}} \begin{pmatrix}
    a b_j\sigma(\beta_j\cdot z + \gamma_j)\sigma(b\cdot G_\mu(z) + c)\\
    ab_j\alpha_jz\sigma'(\beta_j\cdot z+\gamma_j)\sigma(b\cdot G_\mu(z) + c)\\
    ab_j\alpha_j\sigma'(\beta_j\cdot z + \gamma_j)\sigma(b\cdot G_\mu(z) + c)
\end{pmatrix}d\nu(\omega).
$$
Bounding the generator \eqref{eq:generator}, we have
\begin{equation}\label{eq:boundG}
|G_{\mu}(z)| \leq  \int |\sigma(z;\theta)|d\mu(\theta) \leq C\left(\int |\alpha|d\mu(\theta) \,\right).
\end{equation}
Using \eqref{eq:boundG}, and that $|a|,\,|b|,\,|c|\le 1$, we readily obtain \eqref{eq:infty1} and \eqref{eq:infty2}. 
Applying the mean value theorem,
\begin{eqnarray}
    &&\nabla_\omega \sigma(x_1;\omega_1)-\nabla_\omega\sigma(x_2;\omega_2)\nonumber\\
    &&\qquad= \begin{pmatrix}
\sigma'(\xi_0) [b_1\cdot x_1 - b_2\cdot x_2 + c_1 - c_2]\\
    a_1x_1\, \sigma''(\xi_1) [(b_1\cdot x + c_1) - (b_2\cdot y + c_2)] + (x_1(a_1-a_2)+(x_1-x_2)a_2)\sigma'(b_2\cdot y + c_2)\\
a_1 \sigma''(\xi_1) [(b_1\cdot x_1 + c_1) - (b_2\cdot x_2 + c_2)] + (a_1-a_2)\sigma'(b_2\cdot x_2 + c_2)
\end{pmatrix},\nonumber
\end{eqnarray}
where $\xi_0,\xi_1$ are points in between $b_1\cdot x + c_1$ and $b_2\cdot y + c_2$. To obtain the estimate for $v^\Omega$, we consider the difference above in two instances $x_1=x_2=x$, and taking $x_1=G_{\mu_1}(z)$ and $x_2=G_{\mu_2}(z)$. Using the triangle inequality, and $\|\sigma\|_{C^2}<\infty$, we can conclude
\begin{eqnarray}
&&|v^\Omega_{(\mu_1,\nu_1)}(\omega_1,z,x)-v^\Omega_{(\mu_2,\nu_2)}(\omega_2,z,x)|\nonumber\\
&&\qquad\leq C\,\left(1+ |x| +G_{\mu_1}(z)+G_{\mu_2}(z)\right) \, \left( |\omega_1-\omega_2| + | G_{\mu_1}(z)-G_{\mu_2}(z)|\right).\nonumber
\end{eqnarray}

To estimate $G_{\mu_1}(z)-G_{\mu_2}(z)$, we consider $\pi$ a coupling between $\mu_1$ and $\mu_2$, and notice that the difference is given by
$$
|G_{\mu_1}(z)-G_{\mu_2}(z)| = \left| \int \sigma(z;\theta)-\sigma(z;\tilde{\theta}) d\pi(\theta,\tilde{\theta})\right| \leq \int |\sigma(z;\theta)-\sigma(z;\tilde{\theta})|\, d\pi.
$$
Estimating, 
$$
|\sigma(z,\theta)-\sigma(z,\tilde{\theta})| \leq C \big(1+(|\alpha|+|\tilde{\alpha}|)(1+|z|)\big)|\theta-\tilde{\theta}|.
$$ 
Applying the Cauchy-Schwarz inequality,
$$
|G_{\mu_1}(z)-G_{\mu_2}(z)|^2 \leq C \left(1+\left(\int |\alpha|^2\;d\mu_1+\int |\tilde{\alpha}|^2\;d\mu_2\right)(1+|z|^2)\right)\int|\theta-\tilde{\theta}|^2\;d\pi.
$$
Taking $\pi$ to be the optimal coupling with respect to the $d_2$ distance, and using \eqref{eq:boundG}, we conclude
$$
\begin{array}{l}
     \ds |v^\Omega_{(\mu_1,\nu_1)}(\omega_1,z)-v^\Omega_{(\mu_2,\nu_2)}(\omega_2,z)|^2\nonumber\\
   \qquad\le C\left(1+ |x|^2+ |z|^2+\sum_{i=1,2}\int |\alpha|^2\;d\mu_i\right)\big(|\omega_1-\omega_2|^2+d^2_2(\mu_1,\mu_2)\big).\nonumber
\end{array}
$$

For $v^\Theta$, apply the same argument as above to obtain a bound that also depends on the size of $|\alpha|$.
\end{proof}

\begin{lemma}\label{lem:growth}
    Let $(\mu,\nu)\in AC([0,T];\mathcal{P}((\R^{L+2})^K)\times \mathcal{P}(Q))$ a weak solution to \eqref{eq:dynamics}, then
    \begin{equation}\label{eq:growth}
        \int |\alpha|^2 d\mu_t \leq C\left(\int |\alpha|^2 d\mu_{\mathrm{in}} + t^2\right).
    \end{equation}
\end{lemma}
\begin{proof}
By the bound $\|(V^\Theta_j)_1\|_\infty\leq C$, we conclude that $|\alpha(t,\theta_{\mathrm{in}})|\leq |\alpha_{\mathrm{in}}| + Ct$, which implies the desired bound.
\end{proof}
A key step in the proof of existence and uniqueness, Proposition~\ref{prop:existence} and Proposition~\ref{Stability PDE}, is the stability of the projected ODE flow.
\begin{lemma}\label{lem:stability1}
We consider $(\mu_1,\nu_1),$ $(\mu_2,\nu_2)\in AC([0,T];\mathcal{P}((\R^{L+2})^K)\times \mathcal{P}(Q))$ that satisfy the growth condition \eqref{eq:growth}. The associated flow maps \eqref{eq:diffeos} satisfy the bounds
\begin{eqnarray}\label{eq:stability1}
\ds |\Phi^t_{(\mu_1,\nu_1)}(\theta_1,\omega_1)-\Phi^t_{(\mu_2,\nu_2)}(\theta_2,\omega_2)|^2 \leq e^{C(\Lambda+|\alpha_1|+|\alpha_2|) t} e^{C t^2}|(\theta_1,\omega_1)-(\theta_2,\omega_2)|^2 \nonumber \\ 
\ds + C e^{C(\Lambda+|\alpha_1|+|\alpha_2|) t} e^{C t^2} \int_0^t C(r) d_2^2(\left(\mu_{1},\nu_{1}\right)(r),\left(\mu_{2},\nu_{2}\right)(r))dr, \nonumber
\end{eqnarray}
where 
\begin{align*}
    \Lambda &= 1 + \left(\int |\alpha_1|^2 d\mu_{1,\mathrm{in}}\right)^{1/2} + \left(\int |\alpha_2|^2 d\mu_{2,\mathrm{in}}\right)^{1/2};\\
    C(r)&=e^{-C|\alpha_!|r}e^{-C|\alpha_2|r} e^{-C\Lambda r} e^{-C r^2/2} (\Lambda +r + |\alpha_1| + |\alpha_2|)
\end{align*}
for some constant $C>0$ depending on $T$. 
\end{lemma}
\begin{proof}
Recall the Lipschitz bounds on Lemma~\ref{lem:auxV} are given by
\begin{align*}
    C_\Theta(\theta_1,\theta_2) &= C \left(1+|\alpha_1|+|\alpha_2| + \left(\int |\alpha|^2d\mu_1\right)^{1/2} + \left(\int |\alpha|^2d\mu_2\right)^{1/2}\right),\\
    C_\Omega &= C\left(1 + \left( \int |\alpha|^2d\mu_1\right)^{1/2} + \left( \int |\alpha|^2d\mu_2\right)^{1/2}\right).
\end{align*}
By the bound $\|(V^\Theta_j)_1\|_\infty\leq C$, we conclude that 
$$
|\alpha(t,\theta_{\mathrm{in}})|\leq |\alpha_{\mathrm{in}}| + Ct.
$$
Combining this with the growth assumption \eqref{eq:growth}, we have
$$
C_\Theta(\theta_{1}(t),\theta_{2}(t))\leq C\left( C_\Theta(\theta_{1,\mathrm{in}},\theta_{2,\mathrm{in}}) + t\right)\quad\mbox{and}\quad C_\Omega(t)\leq C\left( C_\Omega(\mu_{1,\mathrm{in}},\mu_{2,\mathrm{in}}) + t\right).
$$

Taking the derivative of the distance, we find
\begin{eqnarray*}
&&\frac{1}{2}\frac{d}{dt} |\theta_1(t)-\theta_2(t)|^2 = \left\langle \theta_1(t)-\theta_2(t),V^\Theta_{(\mu_{t,1},\nu_{t,1})}(\theta_1(t)) - V^\Theta_{(\mu_{t,2},\nu_{t,2})}(\theta_2(t))\right\rangle\\
&&\qquad\leq C_\Theta(\theta_{1}(t),\theta_{2}(t))\left( |\theta_1(t) - \theta_2(t)|^2 + d_2(\mu_{1,t},\mu_{2,t})^2 + d_2(\nu_{1,t},\nu_{2,t})^2\right),\\
&&\frac{1}{2}\frac{d}{dt} |\omega_1(t)-\omega_2(t)|^2\\
&&\qquad= \left\langle \omega_1(t)-\omega_2(t),\mathrm{Proj}_{\pi(Q)}V^\Omega_{(\mu_{t,1},\nu_{t,1})}(\omega_{1}(t))-\mathrm{Proj}_{\pi(Q)}V^\Omega_{(\mu_{t,2},\nu_{t,2})}(\omega_{2}(t))\right\rangle\\
&&\qquad\le \left\langle \omega_1(t)-\omega_2(t),V_{(\mu_{t,1},\nu_{t,1})}^\Omega(\omega_{1}(t)) - V_{(\mu_{t,2},\nu_{t,2})}^\Omega(\omega_{2}(t))\right\rangle\\
&&\qquad\le C_\Omega(t)\left( |\omega_1(t) - \omega_2(t)|^2 + d_2(\mu_{1,t},\mu_{2,t})^2\right),
\end{eqnarray*}
where we have used the non-expansiveness property of the projection.

Let $\Lambda_{\Omega} = 1 + \left(\int |\alpha|^2 d\mu_{1,\mathrm{in}} \right)^{1/2} + \left(\int |\alpha|^2 d\mu_{2,\mathrm{in}}\right)^{1/2}$ and $\Lambda_{\Theta} = \Lambda_\Omega + |\alpha_1(0)| + |\alpha_2(0)|$. The estimates above can then be written as
\begin{eqnarray*}
\frac{d}{dt} |\theta_1(t)-\theta_2(t)|^2 &\leq& C(\Lambda_\Theta + t)\left( |\theta_1(t) - \theta_2(t)|^2 + d_2(\mu_{1,t},\mu_{2,t})^2 + d_2(\nu_{1,t},\nu_{2,t})^2\right),\\
\frac{d}{dt} |\omega_1(t)-\omega_2(t)|^2 &\leq& C(\Lambda_\Omega + t)\left( |\omega_1(t) - \omega_2(t)|^2 + d_2(\mu_{1,t},\mu_{2,t})^2\right),
\end{eqnarray*}
which by Gronwall's inequality implies that
\begin{eqnarray*}
&&|\theta_1(t)-\theta_2(t)|^2 \leq e^{C(\Lambda_{\Theta} t + t^2)}|\theta_{1,\mathrm{in}}-\theta_{2,\mathrm{in}}|^2\\
&&\qquad+ C \int_0^t e^{C(\Lambda_{\Theta} t + t^2 - \Lambda_{\Theta} r - r^2)}(\Lambda_\Theta + r) \left(d_2^2((\mu_{1},\nu_{1}),(\mu_{2},\nu_{2})\right) dr.\\
&&|\omega_1(t)-\omega_2(t)|^2 \leq e^{C(\Lambda_{\Omega} t + t^2/2)} |\omega_{1,\mathrm{in}}-\omega_{2,\mathrm{in}}|^2\\
&&\qquad+C\int_0^t e^{C(\Lambda_{\Omega} t + t^2/2 - \Lambda_{\Omega}r - r^2)}(\Lambda_\Omega + r) d_2^2(\mu_{1},\mu_{2}) dr.
\end{eqnarray*}
Putting both inequalities together, we arrive at the desired result.
\end{proof}

We now use this ODE estimate to prove Proposition~\ref{Stability PDE}.
\begin{proof}[Proof of Proposition~\ref{Stability PDE}]
Let
$$
d(t) = d_2^2(\left(\mu_{1},\nu_{1}\right)(t),\left(\mu_{2},\nu_{2}\right)(t)),
$$
and notice that for any coupling $\Pi_*$ between $\mu_{1,\mathrm{in}}\otimes \nu_{1,\mathrm{in}}$ and $\mu_{2,\mathrm{in}}\otimes \nu_{2,\mathrm{in}}$
$$
d(t)\le \int |(\theta_1(t),\omega_1(t))-(\theta_2(t),\omega_2(t))|^2d\Pi_*((\theta_{\mathrm{1,in}},\omega_{\mathrm{1,in}}),(\theta_{\mathrm{2,in}},\omega_{\mathrm{2,in}})),
$$
since the push-forward of $\Pi_*$ along the ODE flow at time $t$ is a coupling between $\mu_{1,t}\otimes \nu_{1,t}$ and $\mu_{2,t}\otimes \nu_{2,t}$. Using Lemma~\ref{lem:stability1}, we obtain that
\begin{eqnarray*}
&&d(t)\le\underbrace{\int e^{C(\Lambda+|\alpha_1|+|\alpha_2|) t} e^{C t^2}|(\theta_1,\omega_1)(0)-(\theta_2,\omega_2)(0)|^2 d\Pi_*}_{I}\\
&&\qquad+ C  \int_0^t d(r) \underbrace{\left(\int e^{C(\Lambda+|\alpha_1|+|\alpha_2|) (t-r)} e^{C (t^2-r^2)}(\Lambda +r + |\alpha_1| + |\alpha_2|)d\Pi_*\right)}_{II}dr .
\end{eqnarray*}

For $I$ we apply the Cauchy-Schwarz and Cauchy's inequality, and take $\Pi_*$ as the optimal coupling with respect to the $4$-Wasserstein to get the bound,
\begin{eqnarray*}
I&\le& e^{C(\Lambda t + t^2)} \left(\int e^{C|\alpha|t} d\mu_{1,\mathrm{in}}+\int e^{C|\alpha|t} d\mu_{2,\mathrm{in}}\right)^{1/2} \left(\int |(\theta_1,\omega_1)(0)-(\theta_2,\omega_2)(0)|^4 d\Pi_*\right)^{1/2}\\
&\le&e^{C(\Lambda t + t^2)} \left(\int e^{C|\alpha|t} d\mu_{1,\mathrm{in}}+\int e^{C|\alpha|t} d\mu_{2,\mathrm{in}}\right)^{1/2} d_4^2(\left(\mu_{\mathrm{in},1},\nu_{\mathrm{in},1}\right),\left(\mu_{\mathrm{in},2},\nu_{\mathrm{in},2}\right)).
\end{eqnarray*}

We bound II from above uniformly in $r$ by the Cauchy-Schwarz inequality,
$$
II \leq C\left(\int e^{C|\alpha|t} d\mu_{1,\mathrm{in}}+\int e^{C|\alpha|t} d\mu_{2,\mathrm{in}}\right)^{1/2}\left(\Lambda_\Omega + t\right).
$$
Therefore, we find that for every $t>0$
$$
d(t) \leq A(t) d_4^2(\left(\mu_{\mathrm{in},1},\nu_{\mathrm{in},1}\right),\left(\mu_{\mathrm{in},2},\nu_{\mathrm{in},2}\right)) + B(t)\int_0^t d(r) dr,
$$
where we define $B(t) = C A(t) (\Lambda + t)$ and
$$
A(t) = e^{C(\Lambda t + t^2)} \left(\int e^{C|\alpha|t} d\mu_{1,\mathrm{in}}+\int e^{C|\alpha|t} d\mu_{2,\mathrm{in}}\right)^{1/2}.
$$
Gronwall's inequality implies that for all $t\geq 0$
$$
d_2^2(\left(\mu_{t,1},\nu_{t,1}\right),\left(\mu_{t,2},\nu_{t,2}\right)) \leq A(t) e^{tB(t)}d_4^2(\left(\mu_{\mathrm{in},1},\nu_{\mathrm{in},1}\right),\left(\mu_{\mathrm{in},2},\nu_{\mathrm{in},2}\right)),
$$
using the monotonicity of $A(t)$ and $B(t)$.
\end{proof}


\section{Continuous time approximation of parameter dynamics}\label{Section Proof of Main Result}

\begin{proof}[Proof of \cref{Approximation of parameters}]
We consider the parameter training algorithm with learning rate $h>0$ and a single hidden layer of $N$ neurons for both the generator and the discriminator neural networks. We denote the parameter values at step $n$ by $(\theta_i^n,\omega_i^n)_{i=1,\ldots,N}$ and the parameter dynamics is
$$
\begin{cases}
    \theta_i^{n+1} = \theta_i^n + \frac{h}{N}\,v^\Theta_{\mu_N^n,\nu_N^n}(\theta_i^n,z_n)\\
    \omega_i^{n+1} = \mathrm{Proj}_Q(\omega_i^n + \frac{h}{N}\,v^\Omega_{\mu_N^n,\nu_N^n}(\omega_i^n,z_n,x_n)),
\end{cases}
$$
where at each step we sample $x_n\sim P_*$ and $z_n\sim\mathcal{N}$ independently, $\mu_N^n$ denotes the empirical measure associated to $\theta_1^n,\ldots,\theta_N^n$ and $\nu_N^n$ the empirical measure associated to $\omega_1^n,\ldots,\omega_N^n$. The parameters are assumed to be initialized by independently  $(\theta_i^0,\omega_i^0)$ by sampling $\mu_{in}\otimes\nu_{in}$. The linear interpolation of the parameters to a continuous time variable $t>0$ with time step $\Delta t = h/N$ will be denoted by $(\theta_i,\omega_i)$, where we let $\theta_i(t_n) = \theta_i^n$ and $\omega_i(t_n) = \omega_i^n$, with $t_n = n \Delta t = n h/N$. We let $\mu$ and $\nu$ be the empirical measures associated to $\theta_1,\ldots,\theta_N$ and $\omega_1,\ldots,\omega_N$. We suppress the dependence on $N$ of the measures for notational simplicity.

We consider the mean field ODE system defined by the expectation of the vector fields over $z$ and $x$
$$
\begin{cases}
    \frac{d}{dt}\hat{\theta}_i = V^\Theta_{(\hat\mu,\hat\nu)}(\hat{\theta_i})\\
    \frac{d}{dt}\hat{\omega}_i = \mathrm{Proj}_{\pi_Q(\omega_i)}V^\Omega_{(\hat\mu,\hat\nu)}(\hat{\omega_i}),
\end{cases}
$$
where $\hat{\mu}$ and $\hat{\nu}$ are the empirical measures associated to $\hat\theta_1,\ldots,\hat\theta_N$ and $\hat\omega_1,\ldots,\hat\omega_N$, respectively, and the initial conditions are coupled to the parameter training by $\hat{\theta}_i(0)=\theta_{i}^0$ and $\hat{\omega}_i(0)=\omega_{i}^0$. More clearly, the probability measures $\hat{\mu}$ and $\hat{\nu}$ are the solutions of the PDE \eqref{eq:dynamics} with random initial conditions chosen as $(\hat\mu(0),\hat\nu(0)) = (\mu_N(0),\nu_N(0))$.

To simplify the arguments, we first consider the distance between mean field ODE system and the discrete projected forward Euler algorithm
$$
\begin{cases}
    \hat{\theta}_i^{n+1}= \hat{\theta}_i^{n}+\Delta t\, V^\Theta_{(\hat\mu^n,\hat\nu^n)}(\hat{\theta_i})\\
    \hat{\omega}^{n+1}_i = \mathrm{Proj}_{Q}\left(\hat{\omega}^{n}_i+\Delta t\, V^\Omega_{(\hat\mu^n,\hat\nu^n)}(\hat{\omega}^n_i)\right)
\end{cases},
$$
where we let $T>0$ be a fixed time horizon and consider $\Delta t = h/N$, where $h>0$ is the user defined learning rate. To estimate the difference between the continuum and the discrete approximation, we can use a similar argument to \cref{thm:approx}, taking into consideration the bound on the Lipschitz constant of the vector fields given by Lemma~\ref{lem:auxV}. We can obtain the bound
$$
\E\left[\frac{1}{N}\sum_{i=1}^N|\hat{\theta}_i^{\Delta t}-\hat{\theta}|^2\right]\le \Delta t C\left(1+\E_{\mu^{in}}\left[e^{Ce^{C|\alpha|}}\right]\right).
$$
The argument is simpler than the argument below, so we skip it to avoid burdensome repetition.

We define 
$$
e^n_i=|\hat{\theta}^n_i-\theta^n_i|^2+|\hat{\omega}^n_i-\omega^n_i|^2\qquad\mbox{and}\qquad e^n=\frac{1}{N}\sum_{i=1}^N e^n_i,
$$
and notice the inequality
$$
d_2^2((\mu^n,\nu^n),(\hat{\mu}^n,\hat{\nu}^n))\le e^n.
$$
Using a step in either algorithm
$$
\begin{array}{rcl}
  \ds e^{n+1}_i&\ds=&|\hat{\theta}^n_i+\Delta t V^\Theta_{(\hat\mu^n,\hat\nu^n)}(\hat{\theta}^n_i)-(\theta^n_i+\Delta t v^\Theta_{(\mu^n,\nu^n)}(\theta^n_i))|^2\\
  &&\ds\qquad+| \mathrm{Proj}_{Q}(\hat{\omega}^{n}_i+\Delta t V^\Omega_{(\hat\mu^n,\hat\nu^n)}(\hat{\omega}^n_i))- \mathrm{Proj}_{Q}(\omega^{n}_i+\Delta t v^\Omega_{(\mu^n,\nu^n)}(\omega^n_i))|^2.
\end{array}
$$
Using that the projection is contractive, expanding the square and bounding  we obtain
$$
e_i^{n+1}\le e_i^n+\Delta t(A_i^n+B_i^n)  +(\Delta t)^2 C_i^n,
$$
where
\begin{eqnarray*}
A_i^n&=&-2\langle\hat{\theta}^n_i-\theta^n_i,V^\Theta_{(\hat\mu^n,\hat\nu^n)}(\hat{\theta}^n_i)-V^\Theta_{(\mu^n,\nu^n)}(\theta^n_i)\rangle
+\langle\hat{\omega}^n_i-\omega^n_i,V^\Omega_{(\hat\mu^n,\hat\nu^n)}(\hat{\omega}^n_i)-V^\Omega_{(\mu^n,\nu^n)}(\omega^n_i)\rangle,\\
B_i^n&=&-2\langle\hat{\theta}^n_i-\theta^n_i,V^\Theta_{(\hat{\mu}^n,\hat{\nu}^n)}(\theta^n_i)-v^\Theta_{(\mu^n,\nu^n)}(\theta^n_i,z_n)\rangle
\\
&&\qquad+2\langle\hat{\omega}^n_i-\omega^n_i,V^\Omega_{(\hat{\mu}^n,\hat{\nu}^n)}(\omega^n_i)-v^\Omega_{(\mu^n,\nu^n)}(\omega^n_i,z_n,x_n)\rangle,
\end{eqnarray*}
and
$$
C_i^n=2\left(|V^\Theta_{(\hat\mu^n,\hat\nu^n)}(\hat{\theta}^n_i)|^2+|v^\Theta_{(\mu^n,\nu^n)}(\theta^n_i)|^2+|V^\Omega_{(\hat\mu^n,\hat\nu^n)}(\hat{\omega}^n_i)|^2+|v^\Omega_{(\mu^n,\nu^n)}(\omega^n_i)|^2\right).
$$
Using the bounds Lemma~\ref{lem:auxV}, we get the growth bound
$$
|\alpha^n_i|,\,|\hat{\alpha}^n_i|\le |\alpha_{i,in}|+C n \Delta t,
$$
and the estimates for $n\Delta t<T$
$$
A^n_i\le K_i (e^n_i+e^n)\qquad\mbox{and}\qquad C^n_i\le (1+|x^n|^2+|z_n|^2)K^2_i,
$$
where
$$
K_i=C\left(1+\left(\frac{1}{N}\sum_j|\alpha_{j,in}|^2\right)^{1/2}+|\alpha_{i,in}|\right).
$$
Using $e_0=0$ and a telescopic sum, we get
\begin{eqnarray*}
    e_{i}^{n+1}&\le& \Delta t K_i \sum_{r=0}^n (e^r_i+e^r)+\Delta t \sum_{r=0}^n B_i^r+ \Delta t K_i^2 \sum_{r=0}^n(1+|x^r|^2+|z_r|^2).
\end{eqnarray*}
Next, we will take the conditional expectation with respect to the variables $\{\alpha_{j,in}\}$. To this end, we notice the bound
\begin{eqnarray*}
    \E\left[\left. \sum_{r=0}^n B_i^r \right|\{\alpha_{j,in}\}\right]&\le& \E\left[\left. \left|\sum_{r=0}^n B_i^r \right|^2\right|\{\alpha_{j,in}\}\right]^{1/2}\\
    &=& \left(\sum_{r=0}^n\E[|B_i^r|^2||\{\alpha_{j,in}\}]+2\sum_{r_1=0}^n\sum_{r_2=r_1+1}^n\E[B_i^{r_1}B_i^{r_2}|\{\alpha_{j,in}\}]\right)^{1/2}\\
    &\le &\left(K_i^2\sum_{r=0}^n\E[e_i^r+e^r|\{\alpha_{j,in}\}]\right)^{1/2}\\
    &\le & K_i\left(1+ \sum_{r=0}^n\E[e_i^r+e^r|\{\alpha_{j,in}\}]\right),
\end{eqnarray*}
where we have used that by Lemma~\ref{lem:auxV}
$$
|B_i^r|^2\le K_i^2 (e^r_i+e^r)
$$
and that
$$
\E[B_i^{r_1}B_i^{r_2}|\{\alpha_{j,in}\}]=0,
$$
which follows by using the law of iterated expectation with the sigma algebra $\mathcal{F}^{r_2}$ generated by  $\{(\theta^i_{in},\omega^{i}_{in}\}_{i=1}^N$, $\{x^r\}_{r=0}^{r_2-1}$ and $\{z^r\}_{r=0}^{r_2-1}$. Namely,
$$
\E[B_i^{r_1}B_i^{r_2}|\{\alpha_{j,in}\}]=\E[B_i^{r_1}\E[B_i^{r_2}|\mathcal{F}^{r_2}]|\{\alpha_{j,in}\}],
$$
where we have used that each $B_i^{r_1}$ is a measure with respect to $\mathcal{F}^{r_2}$ as $r_1<r_2$. Finally, using that $z^{r_2}$ and $x^{r_2}$ are independent with respect $\mathcal{F}^{r_2}$, and that $(\hat{\theta}_i^{r_2-1},\hat{\omega}^{r_2-1}_i)$ and $(\theta_i^{r_2-1},\omega^{r_2-1}_i)$ are measurable with respect $\mathcal{F}^{r_2}$, we have
\begin{eqnarray*}
\E[B_i^{r_2}|\mathcal{F}^{r_2}]&&=-2\E_{z^{r_2}}[\langle\hat{\theta}^{r_2-1}_i-\theta^{r_2-1}_i,V^\Theta_{(\hat{\mu}^{r_2-1},\hat{\nu}^{r_2-1})}(\theta^{r_2-1}_i)-v^\Theta_{(\mu^{r_2-1},\nu^{r_2-1})}(\theta^{r_2-1}_i,z^{r_2})\rangle]\\
&&+2\E_{z^{r_2},x^{r_2}}[\langle\hat{\omega}^{r_2-1}_i-\omega^{r_2-1}_i,V^\Omega_{(\hat{\mu}^{r_2-1},\hat{\nu}^{r_2-1})}(\omega^n_i)-v^\Omega_{(\mu^{r_2-1},\nu^{r_2-1})}(\omega^{r_2-1}_i,z^{r_2},x^{r_2})\rangle]\\
&&=0.    
\end{eqnarray*}

Using the previous bound, that the distributions for $x^r$ and $z^r$ have finite second moments, and that $K_i$ is a deterministic function of $\{\alpha_{j,in}\}$, we obtain up to a change of constant
\begin{eqnarray*}
\E[e_{i}^{n+1}|\{\alpha_{j,in}\}]\le \Delta t K_i\sum_{r=0}^n \E[e^r_i+e^r|\{\alpha_{j,in}\}]+K^2_i\Delta t.
\end{eqnarray*}
Applying a discrete version of Gromwal's inequality, we have
\begin{eqnarray*}
\E[e_{i}^{n+1}|\{\alpha_{j,in}\}]\le  \Delta t K_i e^{TK_i} \sum_{r=0}^{n} \E[e^r|\{\alpha_{j,in}\}]+\Delta t K^2_ie^{TK_i}.
\end{eqnarray*}
Summing over $i$, we obtain
\begin{eqnarray*}
\E[e^{n+1}|\{\alpha_{j,in}\}]\le  \Delta t K \sum_{r=0}^{n} \E[e^r|\{\alpha_{j,in}\}]+\Delta t \frac{1}{N}\sum_{i=1}^N K^2_ie^{TK_i},
\end{eqnarray*}
where
$$
K=\frac{1}{N}\sum_{i=1}^N K_ie^{TK_i}.
$$
Using discrete Gromwall's inequality one last time we have the estimate
\begin{equation}
\E[e^{n+1}|\{\alpha_{j,in}\}]\le \Delta t e^{TK}\frac{1}{N}\sum_{i=1}^N K^2_ie^{TK_i} .    
\end{equation}
Taking expectation, we can bound
$$
\E[e^{n+1}]\le \Delta t \left(\E e^{2TK}+\frac{1}{N}\sum_{i=1}^N K_i^4+\frac{1}{N}\sum_{i=1}^N e^{2TK_i}\right).
$$
Hence, up to changing constants we have the bound 
\begin{eqnarray}
\E[e^{n+1}]&\le& \Delta t C\left(1+ \E \left[e^{C\frac{1}{N}\sum_{i=1}^Ne^{C|\alpha_i|}}\right]\right)\nonumber\\
&=&\Delta t C\left(1+\E_{\mu^{in}}\left[e^{\frac{C}{N}e^{C|\alpha|}}\right]^{N}\right)\nonumber\\
&\le& \Delta t C\left(1+\E_{\mu^{in}}\left[e^{Ce^{C|\alpha|}}\right]\right)\le C \Delta t= C \frac{h}{N}.
\end{eqnarray}
The desired bound \eqref{eq:nocurseofdimensionality} follows from using the bound \eqref{eq:expmoments} to show that the right hand side above is finite.
\end{proof}

\section{Conclusions and future directions}\label{sec:conclusion}
We showed rigorously and quantitatively that the Wasserstein-GAN algorithm is a stochatic discretization of the well-posed PDE system given by \eqref{eq:dynamics}. Here, we use the insight gained from the dynamics to explain some of the pitfalls of W-GAN \cite{arjovsky2017wasserstein} that help explain why is the algorithm finicky to converge. We center in two salient points: the discontinuity of the vector field for the parameters of the discriminator network  and the long time behavior of the mean field dynamics.

We noticed that the clipping of the parameters induces that the dynamics are given by a discontinuous vector field, which forces the dynamics into a box $Q$. In essence, the parameters of the discriminator move within the box $Q$ without anticipating its boundary and crash into $\partial Q$. This is akin to birds flying into a window. This produces blow-up of the distribution of discriminator parameters in finite time. Still, the measure valued solution is well defined for all times $t>0$. Most noticeably, for this solution once the dimension of the support of the measure is reduced, it will never fatten back up. In an extreme case, the dynamics can lead to the distribution of the discriminator parameters being $\nu(t)=\delta_{\omega(t)}$ for any $t>t_*$.

In the follow up work \cite{gulrajani2017improved}, finite time blow-up was already observed in toy numerical examples. \cite{gulrajani2017improved} improves the original W-GAN algorithm by enforcing 1-Lipschitz condition with a penalization. With respect to the underlying energy functional, this is equivalent for the mean field dynamics to considering
\begin{eqnarray*}
    E[\mu,\nu]&=& \int_{\R^L} D_\nu( G_\mu(z)) \;d\mathcal{N}(z)-\int_{\R^K} D_\nu(x)\;dP_*(x)\\
    &&\qquad+\lambda \int_0^1\int_{\R^L}\int_{\R^K} \left||\nabla D_\nu((1-s)G_\mu(z)+sx)|-1\right|^2\;dP_*(x)d\mathcal{N}(z)ds,
\end{eqnarray*}
with $\lambda$ being a user chosen penalization parameter. The evolution of the mean field limit can be formally characterized as the gradient descent of $E$ on $\mu$ and gradient ascent on $\nu$. In terms of equations we consider
\begin{equation*}
\begin{cases}
    \partial_t \mu-\nabla_\theta\cdot\left(\mu\; \nabla_\theta \frac{\delta E}{\delta \mu}[\mu,\nu]\right)=0,\\
    \partial_t \nu+ \gamma_c \nabla_\omega\cdot\left(\nu\;\nabla_\omega \frac{\delta E}{\delta \nu}[\mu,\nu]\right)=0,\\
    \mu(0)=\mu_{in}, \qquad \nu(0)=\nu_{in}.
\end{cases}
\end{equation*}
Understanding, the difference in the dynamics for these improved algorithms is an interesting open problem.

For the long time behavior of the dynamics \eqref{eq:dynamics}, we refer to Section~\ref{sec:modecollapse} for intuition where we show in a toy example of ODEs that for any initial conditions the dynamics stabilize to a limiting periodic orbit. Generalizing this to absolutely continuous initial data is quite complicated, we mention the recent work for the Euler equation, in \cite{hassainia2023rigorous} the authors construct vortex patches that replicate the motion of leapfrogging vortex points. Moreover, for the general system, we expect that the dynamics will always converge to some limiting periodic orbit. Showing this rigorously is a challenging PDE problem.

In terms of the curse of dimensionality exhibited in Corollary~\ref{corollary}, an alternative would be to quantify the convergence of the algorithm in a Reproducing Kernel Hilbert Space (RKHS). In PDE terms, this would mean to show well posedness of the PDE in a negative Sobolev space like $H^{-s}$ with $s>d/2$.

\section*{Acknowledgments}
We would like to acknowledge Justin Sirignano, Yao Yao and Federico Camara Halac for useful conversations at the beginning of this project. MGD would like to thank the Isaac Newton Institute for Mathematical Sciences, Cambridge, for support and hospitality during the program \textit{Frontiers in Kinetic Theory} where part of the work on this paper was undertaken. This work was supported by EPSRC grant no EP/R014604/1. The research of MGD was partially supported by NSF-DMS-2205937 and NSF-DMS RTG 1840314. The research of RC was partially supported by NSF-DMS RTG 1840314.
\appendix
\section{}\label{app}
 Following the ideas of \cite{henry1973existence}, in this section we prove the existence, uniqueness and stability to a class of ODEs with discontinuous forcing given by a projection. We also show quantitative convergence of the projected forward Euler algorithm, for which we could not find a good reference for.

 Before we present the main result, we introduce some notation that we need. For any closed convex subset $Q\subset\R^d$ and $x\in\R^d$ there exists an unique $\mathrm{Proj}_Q x\in Q$ such that
$$
\|\mathrm{Proj}_Q x-x\| = \min_{q\in Q}\|q-x\|.
$$
The map $\mathrm{Proj}_Q$ is non-expansive, which means that for all $x,\,y\in\R^d$:
$$
\|\mathrm{Proj}_Q(x)-\mathrm{Proj}_Q(y)\| \leq \|x-y\|.
$$
We denote by $\pi_Q(x)\subset\R^d$ the tangent cone of $Q$ at $x\in Q$,
$$
\pi_Q(x) = \overline{\{v\in\R^d\,|\, \exists\, \epsilon>0,\ \ x + \epsilon v\in Q \} } = \left\{v\in\R^d\,\left|\, \lim_{h\to 0^+}\frac{d(x+hv,Q)}{h} = 0\right.\right\},
$$
which is a closed convex cone. The map $\mathrm{Proj}_{\pi_Q(x)}:\R^d\to\R^d$ denotes the projection onto $\pi_Q(x)\subset\R^d$.  We notice that for a smooth vector field $V(x):Q\to \R^d$, the mapping $x\in\R^n\mapsto \mathrm{Proj}_{\pi_Q(x)}(V(x))$ is discontinuous at points $x$ such that $V(x)\notin\pi_Q(x)$. 

\begin{theorem}[\cite{henry1973existence}]\label{existence}
Let $Q\subset\R^d$ be a closed and convex subset of $\R^d$ and $V:\R^d\to\R^d$ a $C^1$ vector field, which satisfies that there exists $C>0$, such that $|V(x)|, |\nabla V(x)|\le C$. Then, for any initial condition $x_{in}\in Q$ there exists a unique absolutely continuous curve $x:[0,\infty)\to Q$ such that
\begin{equation}
\label{Projected ODE}
\begin{cases}
    \dot{x} = \mathrm{Proj}_{\pi_Q(x)}V(x),\\
    x(0)=x_{in},
\end{cases}
\end{equation}
with the equality satisfied for almost every $t$. Moreover, the solutions are also stable with respect to the initial condition $x_{in}$:
$$
\|x_1(t)-x_2(t)\|\leq e^{\|\nabla V\|_{\infty}\,t}\|x_1(0)-x_2(0)\|,
$$
where $x_1$ and $x_2$ are two solution to \eqref{Projected ODE}.
\end{theorem}
Moreover, we can approximate these solutions by a projected forward Euler algorithm.
\begin{theorem}\label{thm:approx}
Let $x_{\Delta t}:[0,\infty)\to Q$ be the linear interpolation at times $n\Delta t$ of $\{x^{n}_{\Delta t}\}$ defined by the projected Euler algorithm
\begin{equation}\label{PEuler}
\begin{cases}
    x^{n+1}_{\Delta t}=\mbox{Proj}_Q( x^{n}_{\Delta t}+\Delta t\, V( x^{n}_{\Delta t}))\\
    x^0_{\Delta t}=x_{in}.
\end{cases}
\end{equation}
Then, for any time horizon $T>0$, as $\Delta t\to 0$ we have 
$$
\|x_{\Delta t}-x\|\le e^{(1+\|\nabla V\|_{\infty})t} \left(2(\Delta t)^{1/2} \|V\|_{\infty}+\Delta t\|V\|_{\infty}\|\nabla V\|_{\infty}\right),
$$
where $x$ is the unique solution of \eqref{Projected ODE}.
\end{theorem}
\begin{proof}[Proof of \cref{existence} and \cref{thm:approx}]
For each $x\in Q$, we define the normal cone $N_Q(x)$ as
$$
N_Q(x) = \{ n\in\R^d\,|\, \forall q\in Q:\ \ \langle n,q-x\rangle\leq 0\},
$$
or equivalently the set of vectors $n\in\R^d$ such that $\langle n,w\rangle\leq 0$ for all $w\in \pi_Q(x)$. It follows directly from the projection property the following useful result.
\begin{lemma}
\label{appendix aux lemma}
    For any $v\in\R^d$ and $x\in Q$, the vector $n_x=v-\mathrm{Proj}_{\pi_Q(x)}v$ is orthogonal to $\mathrm{Proj}_{\pi_Q(x)}v$ and $n_x\in N_Q(x)$. Conversely, if $w\in \pi_Q(x)$ is that $n_x=v-w\in N_Q(x)$ and $\langle w,v-w\rangle = 0$, then $w=\mathrm{Proj}_{\pi_Q(x)}v$.
\end{lemma}

\noindent\textbf{Uniqueness and Stability.} Suppose that $x_1:[0,T]\to Q$ and $x_2:[0,T]\to Q$ are solutions of \eqref{Projected ODE}. Then,
$$
\frac{d}{dt} \frac{1}{2}\|x_1-x_2\|^2 = \langle x_1-x_2,\mathrm{Proj}_{\pi_Q(x_1)}V(x_1)-\mathrm{Proj}_{\pi_Q(x_2)}V(x_2)\rangle.
$$
Using the property of the projection we have
\begin{eqnarray*}
\langle x_1-x_2,\mathrm{Proj}_{\pi_Q(x_1)}V(x_1)-\mathrm{Proj}_{\pi_Q(x_2)}V(x_2)\rangle&\leq& \langle x_1-x_2,V(x_1)-V(x_2)\rangle\\
&\leq& \|\nabla V\|_{\infty} \,\|x_1-x_2\|^2, 
\end{eqnarray*}
where we have used Lemma~\ref{appendix aux lemma} for the first inequality, and the Lipschitz propety for the second inequality. Grownwall's inequality applied to $\|x_1-x_2\|^2$ gives:
$$
\|x_1(t)-x_2(t)\|\leq e^{\|\nabla V\|_{\infty}\,t}\|x_1(0)-x_2(0)\|,
$$
which shows the uniqueness and stability of solutions with respect to the initial condition.

\noindent\textbf{Equivalence with a relaxed problem.} 
Using $N_Q(x)$, we now introduce a relaxed problem which we prove is equivalent to the ODE \eqref{Projected ODE}. For each $x\in Q$ we define the compact convex set $\mathcal{V}(x)\subset\R^d$ by
$$
\mathcal{V}(x) = \{ V(x) - n_x\,|\, n_x\in N_Q(x),\ \ \|n_x\|^2\leq V(x)\cdot n_x\}.
$$
The relaxed problem is finding an absolutely continuous curve $x:[0,T]\to Q$ such that 
\begin{equation}\label{relaxed}
    \begin{cases}
    \dot{x}(t)\in \mathcal{V}(x(t))\\
    x(0)=x_{in},
\end{cases}
\end{equation}
for almost every $t\in [0,T]$. To show the equivalence between \eqref{Projected ODE} and \eqref{relaxed}, we need the following Lemma.
\begin{lemma}\label{lem:2}
    For all $x\in Q$ we have $V(x)\in\mathcal{V}(x)$, $\mathrm{Proj}_{\pi_Q(x)}V(x)\in\mathcal{V}(x)$ and
$$
\mathcal{V}(x)\cap\pi_{Q}(x) = \{ \mathrm{Proj}_{\pi_Q(x)}V(x)\}.
$$
\end{lemma}
\begin{proof}[Proof of Lemma~\ref{lem:2}]
    Taking $n_x=0$ in the definition of $\mathcal{V}(x)$ gives that $V(x)\in \mathcal{V}(x)$. Writing $n_x=V(x)-\mathrm{Proj}_{\pi_Q(x)}V(x)$, we recall from Lemma~\ref{appendix aux lemma} that $n_x\in N_Q(x)$ and 
    $$
    \langle n_x,\mathrm{Proj}_{\pi_Q(x)}V(x)\rangle=0,
    $$
    so $\|n_x\|^2 = \langle n_x,V(x)\rangle$ and we conclude $\mathrm{Proj}_{\pi_Q(x)}V(x)\in\mathcal{V}(x)$. Now note that if $V(x)-n_x\in \pi_Q(x)$ with $n_x\in N_Q(x)$, then $\langle V(x)-n_x,n_x\rangle \leq 0$ with equality only if $V(x)-n_x=\mathrm{Proj}_{\pi_Q(x)}V(x)$, as we have noted above. So if $V(x)-n_x\in\pi_Q(x)\cap\mathcal{V}(x)$ then $V(x)-n_x=\mathrm{Proj}_{\pi_Q(x)}V(x)$. 
\end{proof}

An absolutely continuous curve $x:[0,T]\to Q$ is such that
$$
\lim_{h\to 0^+} \frac{x(t+h)-x(t)-\dot{x}(t)\,h}{h}=0,
$$
for almost every $t$. Since $x(t)\in Q$ for all $t\in [0,T]$,
$$
0=\lim_{h\to 0^+}\frac{d(x(t+h),Q)}{h} = \lim_{h\to 0^+}\frac{d(x(t)+h\dot{x}(t),Q)}{h},
$$
which shows $\dot{x}(t)\in \pi_Q(x(t))$. If we have a solution to the relaxed problem, then the differential inclusion $\dot{x}(t)\in \mathcal{V}(x(t))$ is satisfied almost everywhere, therefore we have $\dot{x}(t) = \mathrm{Proj}_{\pi_Q(x)}V(x)$ since by Lemma~\ref{lem:2} $\mathcal{V}(x)\cap \pi_Q(x)=\{\mathrm{Proj}_{\pi_Q(x)}V(x)\}$, and we conclude that \eqref{relaxed} and \eqref{Projected ODE} are equivalent.

\noindent\textbf{Existence.} Consider $\nu_Q(x^{n+1}_{\Delta t})\in N_Q(x^{n+1}_{\Delta t})$ unit vectors and $0\leq\lambda\leq 1$ such that
$$
x^{n+1}_{\Delta t}=x^{n}_{\Delta t}+\Delta t\, V( x^{n}_{\Delta t})-\Delta t\, \lambda\, (V( x^{n}_{\Delta t})\cdot \nu_Q(x^{n+1}_{\Delta t}))_+ \nu_Q(x^{n+1}_{\Delta t}),
$$
which follows directly from the properties of the projection. For each $n\geq 0$ we consider the discrete velocity
$$
u_{\Delta t}^n = \frac{x^{n+1}_{\Delta t}-x^{n}_{\Delta t}}{\Delta t},
$$
which we re-write as
\begin{eqnarray*}
u_{\Delta t}^n &=& V(x^{n+1}_{\Delta t})- \lambda (V( x^{n+1}_{\Delta t})\cdot \nu_Q(x^{n+1}_{\Delta t}))_+ \nu_Q(x^{n+1}_{\Delta t})+\underbrace{V(x^{n}_{\Delta t})-V(x^{n+1}_{\Delta t})}_{I} \\
&&\qquad + \underbrace{\lambda\, ((V( x^{n}_{\Delta t})\cdot \nu_Q(x^{n+1}_{\Delta t}))_+-(V( x^{n+1}_{\Delta t})\cdot \nu_Q(x^{n+1}_{\Delta t}))_+) \nu_Q(x^{n+1}_{\Delta t})}_{II}.
\end{eqnarray*}
We notice the bounds
$$
|I|,\;|II|\le \|\nabla V\|_{\infty} |x^{n+1}_{\Delta t}-x^n_{\Delta t}|\le \Delta t \|\nabla V\|_{\infty}\|V\|_{\infty}. 
$$
Therefore, letting $B_1$ denote the unit ball centred at the origin,
$$
u_{\Delta t}^n\in \mathcal{V}(x^{n+1}_{\Delta t})+\|\nabla V\|_{\infty}\|V\|_{\infty}\Delta t\, B_1.
$$
Hence, for any $\Delta t>0$ we can conclude that for a.e. $t$
\begin{equation}\label{eq:estimate}
(x_{\Delta t}(t),\dot{x}_{\Delta t}(t))\in \mathrm{Graph}(\mathcal{V})+ \Delta t( \|V\|_{\infty}B_1\times \|\nabla V\|_{\infty}\|V\|_{\infty}  B_1),    
\end{equation}
Noting that $x_{\Delta t}$ is uniformly Lipschitz with constant less that $\|V\|_\infty$, we get up to subsequence there exists a Lipschitz function $X:[0,\infty)\to Q$ such that $x_{\Delta t}\to X$ uniformly at compact subintervals, by Arzerlà-Ascoli. We conclude using Mazur's Lemma that the derivative of $x$ belongs almost everywhere to the upper limit of the convex hull of the values of $\dot{x}_{\Delta t}(t)$,
$$
\dot{x}(t)\in \limsup_{\epsilon \to 0^+} \mathrm{co}({\dot{x}_{\Delta t}(t)}_{0< \Delta t<\epsilon}).
$$
Using that $\mathcal{V}(x(t))$ is convex and closed, we conclude
$$
(x(t),\dot{x}(t))\in \mathrm{Graph}(\mathcal{V}),
$$
which implies that $x$ is a solution to the relaxed problem, and therefore a solution to the original \eqref{Projected ODE}.

\noindent\textbf{Quantitative Estimate.} We differentiate the distance, between $X$ and $x_{\Delta t}$ to obtain 
\begin{eqnarray*}
\frac{1}{2}\frac{d}{dt}|X-x_{\Delta t}|^2&=&\langle X-x_{\Delta t}, \dot{X}-\dot{x}_{\Delta t}\rangle\\
&\le& \langle X-x_{\Delta t}, V(X)-V({x}_{\Delta t})\rangle+ \Delta t \|V\|_{\infty} |\dot{X}-\dot{x}_{\Delta t}|\\
&&\qquad+\Delta t \|V\|_{\infty}\|\nabla V\|_{\infty} |X-x_{\Delta t}|\\
&\le & (1+\|\nabla V\|_{\infty})|X-x_{\Delta t}|^2+2\Delta t \|V\|^2_{\infty}+(\Delta t)^2\|V\|^2_{\infty}\|\nabla V\|^2_{\infty},
\end{eqnarray*}
where we have used estimate \eqref{eq:estimate} and the contraction to property. Using Gromwall's inequality and that $|X-x_{\Delta t}|^2=0$ , we obtain

\begin{equation*}
|X-x_{\Delta t}|^2\le e^{2(1+\|\nabla V\|_{\infty})t} \left(2\Delta t \|V\|^2_{\infty}+(\Delta t)^2\|V\|^2_{\infty}\|\nabla V\|^2_{\infty}\right).    
\end{equation*}
\end{proof}


\bibliography{biblio}

\end{document}